\documentclass[absolute]{mymathart}
\usepackage[theorems]{mymathmacros}
\usepackage[usenames,dvipsnames]{color}
\usepackage{subfig}
\usepackage[shortlabels]{enumitem}
\usepackage{hyperref}
\usepackage{amssymb}
\usepackage{graphicx}
\usepackage{epigraph}
\usepackage{amsfonts,amssymb,color}
\usepackage[mathscr]{eucal}
\usepackage{amsmath, amsthm, amscd}
\usepackage{mathrsfs}
\usepackage{amsbsy}
\usepackage{float}
\usepackage{verbatim}
\usepackage{amsfonts} 
\usepackage{pifont}

\title[On the Hausdorff dimensions of exploding measures originated from IFS]{On the Hausdorff dimensions of exploding measures originated from IFS}

\author{Liangang Ma}

\address{Dept.\ of Mathematical Sciences, Binzhou University, Huanghe 5th Road No. 391, Binzhou 256600, Shandong, P. R. China} 
\email{maliangang000@163.com}
\thanks{The work is supported by ZR2019QA003 from SPNSF and 12001056 from NSFC}    
\newtheorem{theorem}[subsection]{Theorem}
\newtheorem{lemma}[subsection]{Lemma}
\newtheorem{Young's Lemma}[subsection]{Young's Lemma}
\newtheorem{MU Theorem}[subsection]{Mauldin-Urba\'nski Theorem}
\newtheorem{SSU Theorem}[subsection]{Simon-Solomyak-Urba\'nski Theorem}
\newtheorem{Prohorov's Theorem}[subsection]{Prohorov's Theorem}
\newtheorem{proposition}[subsection]{Proposition}
\newtheorem{corollary}[subsection]{Corollary}
\newtheorem{definition}[subsection]{Definition}
\newtheorem{exm}[subsection]{Example}
\newtheorem{rem}[subsection]{Remark}

\newcommand{\bm}[1]{\mbox{\boldmath{$#1$}}}

\numberwithin{equation}{section}

\DeclareMathOperator*{\esssup}{ess\,sup}

\begin{document} 

\begin{abstract}
   
   We study continuity and discontinuity properties of some popular measure-dimension mappings under some topologies on the space of probability measures in this work. We give examples to show that no continuity can be guaranteed under general weak, setwise or TV topology on the space of measures for any of these measure-dimension mappings. However, in some particular circumstances or by assuming some restrictions on the measures, we do have some (semi-)continuity results.  We then apply our continuity results to concerning measures appearing in two kinds of infinite iterated function systems, namely, CIFS and CGDMS , to show the convergence of the Hausdorff dimensions of the concerning measures induced from the finite sub-systems of these infinite systems. These applications  answer  a problem of Mauldin-Urba\'nski originally posed on $t$-conformal measures for CIFS in the last 90s positively.  Finally we indicate more applications of our techniques in some more general circumstances and give some remarks on the relationship between the Hausdorff dimensions of measures and their logarithmic density in general settings.

 \end{abstract}
 
 \maketitle

\section{Introduction}\label{sec1}

Let $X$ be a metric space endowed with a metric $\rho$. Denote by $\mathcal{M}(X)$ to be the collection of all the probability measures on $(X,\mathscr{A})$ with Borel $\sigma$-algebra $\mathscr{A}$. By normalizations if necessary most of our notions and results have their extensions onto the space of finite measures $\tilde{\mathcal{M}}(X)$. There are seldom escape or attraction of mass for our interested measures, except some special cases, for example, in our Lemma \ref{lem11}.  Let $T: X\rightarrow X$ be a transformation. Denote by $\mathcal{M}_\sigma(X)$ and $\mathcal{M}_e(X)$ respectively to be the collections of all the invariant and ergodic probability measures on $(X,\mathscr{A})$, with respect to $T$.  For a set $A\subset X$, let $HD(A)$ be its Hausdorff dimension. 

\begin{definition}
For a measure $\nu\in\mathcal{M}(X)$,  the \emph{lower} and \emph{upper measure-dimension mappings} from $\mathcal{M}(X)$ to $[0,\infty)$ are defined respectively to be:

\begin{center}
$dim_H^0 \nu=\inf\{HD(A): \nu(A)>0, A\in\mathscr{A}\},$  
\end{center}
and
\begin{center}
$dim_H^1 \nu=\inf\{HD(A): \nu(A)=1, A\in\mathscr{A}\}.$
\end{center}
\end{definition}

Generally, the two measure-dimension mappings measure how well the mass is distributed over the space $X$ (see \cite[Section 1.3]{Fal1} by Falconer for a general description of the idea of mass distributions). We have the obvious relationship that 
\begin{center}
$dim_H^0 \nu\leq  dim_H^1 \nu$
\end{center}
for any $\nu\in \mathcal{M}(X)$. For different aims and tastes, people usually pay attention to only one kind of the meterage for their measures interested, while the two measure-dimension mappings can in fact take exact values of big differences from each other at certain measures in $\mathcal{M}(X)$. However, they two coincide with each other on $\mathcal{M}_e(X)$, if the transformation $T$ is equipped with some appropriate restrictions. The transformation $T: X\rightarrow X$ is said to be \emph{inverse-dimension-expanding} if there exists some $A\in\mathscr{A}$, such that 
\begin{center}
$HD(T^{-1}(A))>HD(A)$.
\end{center}

\begin{Young's Lemma}
For a transformation $T$ on $(X,\mathscr{A})$ such that $T$ is not inverse dimension-expanding, then we have
\begin{center}
$dim_H^0 \nu=  dim_H^1 \nu$
\end{center}
for any ergodic measure $\nu\in \mathcal{M}_e(X)$ with respect to $T$.
\end{Young's Lemma} 
\begin{proof}
See \cite[P115]{You}.
\end{proof}

P. Mattila, M. Mor\'an and J. M. Rey made a systematic  study of the desiring properties of the  measure-dimension mappings  $dim_H^0$,  $dim^1_H$, the correlation measure-dimension mapping $dim_C$ and the modified correlation measure-dimension mapping $dim_{MC}$ (see Definition \ref{def5} and \ref{def6}) in \cite{MMR}. However, no continuity property is discussed there. Endowing the space $\mathcal{M}(X)$ (respectively, $\mathcal{M}_\sigma(X)$ or $\mathcal{M}_e(X)$) with the weak, setwise or total variation (TV) topology, we construct some examples of sequences of measures convergent under these topologies  whose sequences of dimensions diverge, or do not converge to the dimensions of their limit measures. We also prove some (semi-)continuity results of the measure-dimension mappings  $dim_H^0$ and  $dim^1_H$ under the setwise topology on  $\mathcal{M}(X)$. The proof of these results are quick, compared with their extensive applications in various situations.

Intuitively, a measure $\nu$ with full dimension $HD(X)$ is best distributed on the ambient space $X$. There are huge amount of work on calculations of the value $dim_H^0 \nu$ or $dim_H^1 \nu$ for various $\nu\in\mathcal{M}(X)$. The local logarithmic density  is a good indicator on the dimension of the measures.  For $x\in X$ and $r\geq 0$, let $B(x,r)$ be the closed ball centered at $x$ with radius $r$. We call
\begin{center}
$\underline{d}(\nu,x)=\liminf_{r\rightarrow 0}\cfrac{\log \nu(B(x,r)) }{\log r}$
\end{center}
and 
\begin{center}
$\overline{d}(\nu,x)=\limsup_{r\rightarrow 0}\cfrac{\log \nu(B(x,r)) }{\log r}$
\end{center}
the \emph{lower} and \emph{upper logarithmic density} of the measure $\nu$ at $x$  respectively. If the two numbers coincide, we call the number $d(\nu,x)$ the \emph{logarithmic  density} of the measure $\nu$ at $x$. Young \cite[Proposition 2.1]{You} gave a sufficient condition to decide the dimension of $\nu$ in terms of its logarithmic density. She proved that, for a measure $\nu\in\mathcal{M}(X)$, if there exists a non-negative real $c$ satisfying 
\begin{center}
 $\nu\{x\in X: d(\nu,x)=c\}=1$,
\end{center}
then it is sure that
\begin{center}
$dim_H^0 \nu=dim^1_H \nu=c$.
\end{center}
We will give an example to show that the condition is not necessary for a measure  $\nu$ to be of dimension $c$ in Section \ref{sec5}. 

The correlation dimension $\dim_C \nu$ of a measure $\nu\in\mathcal{M}(X)$ highlights the average logarithmic density of the measure $\nu$. The notion is introduced by P. Grassberger and I. Procaccia. It has the advantage of being more easily calculated than   $dim_H^0 \nu$ and $dim_H^1 \nu$ \cite{GP1, GP2}. It is widely used in describing sizes of  realistic fractal-like objects, for example, a real network, see \cite{Ros}. 

\begin{definition}\label{def5}
For a measure $\nu\in\mathcal{M}(X)$, the \emph{correlation dimension} of $\nu$ (see \cite{Cut}) is defined to be

\begin{center}
$dim_C \nu=\lim_{r\rightarrow 0}\cfrac{\log\int\nu(B(x,r))d\nu}{\log r},$
\end{center}
in case the above limit exists.
\end{definition}

According to \cite{MMR}, it fails the properties of '(a) Monotonicity' and '(d) Absolutely continuous measures'. Pesin \cite{Pes} modified the above definition to the following one.

\begin{definition}\label{def6}
For a measure $\nu\in\mathcal{M}(X)$, the \emph{modified correlation dimension} $dim_{MC}$ of $\nu$ is 
\begin{center}
$dim_{MC} \nu=\lim_{\delta\rightarrow 0}\sup_{\{A\in\mathscr{A}: \nu(A)\geq 1-\delta\}}\lim_{r\rightarrow 0}\cfrac{\log\int_A\nu(B(x,r))d\nu}{\log r}$,
\end{center}
 in case the above limit exists. 
\end{definition}

It satisfies the above two properties '(a)' and '(d)' which the correlation measure-dimension mapping fails, according to \cite[Theorem 2.10]{MMR}.  

\begin{rem}
Regardless of the existence of the above two limits, one can define the \emph{upper} and \emph{lower correlation dimension} as well as the \emph{upper} and \emph{lower modified correlation dimension} as \cite[(2),(4)]{MMR}. We adopt these definitions directly because all the measures in our concerns in this work have the same upper and lower (modified) correlation dimensions.
\end{rem} 

By \cite[Proposition 10.2]{Fal3} and \cite[Lemma 2.8]{MMR}, we have 
\begin{center}
$dim_H^0 \nu\leq dim_C \nu\leq dim_{MC} \nu$
\end{center}
for any $\nu\in \mathcal{M}(X)$.
 
We will give examples to show that, unfortunately, no semi-continuity can be guaranteed for the two measure-dimension mappings $dim_C$ and $dim_{MC}$ under the most strict topology-the TV topology on $\mathcal{M}(X)$.

Now we transfer our attention to IFS. For a countable index set $I$ with at least two elements, we call the family 
\begin{center}
$S=\{s_i: X\rightarrow X\}_{i\in I}$ 
\end{center}
consisting of (contractive) maps an \emph{iterated function system} (IFS). There are two main focus in the IFS theory. One is on the \emph{limit set} (\emph{attractor})  $J$ of the system $S$, another is on the measures supported on $J$. People are particularly interested in invariant or ergodic ones whose Hausdorff dimensions are of full dimension $HD(J)$.  

There are various distinctions between cases of $\#I$ (the cardinality of $I$) being finite and infinite, on both the topology and geometry of the limit set $J$ as well as the measures supported on $J$. See for example \cite{GM, GMW, HMU, MPU, MU1, MU2, MU3, MiU1, MiU2, RU1, RU2, UZ}. Infinite IFS contain many important systems, for example, the continued-fraction system.

One of the main difficulty arising in the infinite IFS theory (in comparison with the finite case), and probably the most thorny one, is the \emph{entropy} (should be interpreted in corresponding contexts) of the measures supported on  the limit set $J$ may explode, that is, being infinite. In the non-overlapping case, a well-known result of Mauldin-Urba\'nski for conformal iterated function systems (CIFS, see Definition \ref{def7}) states the following.

\begin{MU Theorem}
For a regular infinite CIFS with the $HD(J)$-conformal measure $m$ and the unique ergodic measure $m^*$ equivalent with $m$ on $J$, we have
\begin{center}
$dim_H^0 m =dim_H^0 m^*=dim_H^1 m =dim_H^1 m^*=HD(J)$,
\end{center}
if the push-forward of the shift map $\sigma$ under projection has finite entropy with respect to the measure $m$ (or $m^*$).
\end{MU Theorem}

Surprisingly, we find that the semi-continuity properties of the measure-dimension mappings $dim_H^0$ and $dim_H^1$ we obtain on the space $\mathcal{M}(X)$ can be applied to overcome the obstacle of explosion of entropy in the CIFS theory. We will show that  Mauldin-Urba\'nski Theorem still holds for conformal measures in CIFS  with infinite entropy, see our Theorem \ref{thm3}. This result solves a  problem posed by Mauldin and Urba\'nski in the last 90s-\cite[Problem 7.3]{MU1}. A corresponding problem in CGDMS (conformal graph directed Markov systems, see Definition \ref{def9}) is also conquered by our method.

The method has its broad applications to (exploding) measures originated from infinite IFS in many other different circumstances, as well as to (exploding) measures on some space with some general dynamical structures (see for example our Corollary \ref{cor8}). It usually involves the following steps.

\begin{enumerate}[1.]
\item Establish the existence of (exploding or non-exploding) concerning measure $\nu$ on the ambient space $X$ with some dynamical structure. This means the concerning measure originated from the infinite system in the IFS case.

\item Find a sequence of non-exploding measures $\{\nu_n\}_{n\in\mathbb{N}}$ converging to the concerning measure $\nu$ under the setwise topology. These measures are usually induced from the finite sub-systems  in the IFS case.

\item Show the existence of ergodic measures $\{\nu_n^*\}_{n\in\mathbb{N}}$ (with respect to some appropriate dynamical structures on $X$) equivalent with the ones $\{\nu_n\}_{n\in\mathbb{N}}$ respectively for any $n\in\mathbb{N}$ (sometimes they coincide with each other).

\item Apply Young's Lemma and our semi-continuity results-Theorem \ref{thm1} and \ref{thm4} to show the convergence of the sequences of dimensions of the  measures $\{\nu_n\}_{n\in\mathbb{N}}$ and $\{\nu_n^*\}_{n\in\mathbb{N}}$ to the dimension of the concerning measure $\nu$.
\end{enumerate}

The organization of the paper is as following. In Section 2 we introduce the three kind of popular topologies on $\mathcal{M}(X)$, namely, the weak, setwise and TV topology as well as our two main (semi-)continuity results-Theorem \ref{thm1} and \ref{thm4} regarding the lower and upper Hausdorff measure-dimension mappings under the setwise topology. Then we give a quick introduction on the CIFS and CGDMS. Between the introduction we present the two main results on the dimensions of our concerning (exploding) measures, these are,  Theorem \ref{thm3} regarding (exploding) measures originated from CIFS and  Theorem \ref{thm13} regarding (exploding) measures originated from CGDMS. Section 3 is devoted to the proofs of the two (semi-)continuity results, while some discontinuity results and examples are also included. All the continuity-discontinuity properties of the four measure-dimension mappings $dim_H^0, dim_H^1, dim_C$ and  $dim_{MC}$ are taken into considerations in the general setting. In the following two sections the readers will see applications of our techniques to various measures originated from CIFS and CGDMS. In Section 4 we apply our continuity theorems to CIFS, to show Theorem \ref{thm3}, which liberates the Mauldin-Urba\'nski Theorem from the shackle of finite entropy.  In Section 5 we apply the continuity results to CGDMS, to show Theorem \ref{thm13}, which is deduced from the convergence of sequence of dimensions of the concerning  measures induced from the  finite sub-systems of the corresponding infinite system.  The crucial new ingredient is the use of Corollary \ref{cor3} on deciding weak convergence of a sequence of measures through tightness of the sequence by Billingsley, which follows from the Prohorov's Theorem. We also take the (modified) correlation dimension into consideration in the above two scenarios. In the last section we first point out more applications of our techniques to measures in $\mathcal{M}(X)$ with some other IFS and  more general dynamical structures on the ambient space $X$. Then we make some discussions on the relationship between  $dim_H^0\nu, dim_H^1\nu$ and the logarithmic density of a measure $\nu$ in general circumstances.

\section{The weak, setwise, TV topology on $\mathcal{M}(X)$, CIFS, CGDMS and the main results}\label{sec2}

To discuss the continuity of the measure-dimension mappings $dim_H^0$ and $dim^1_H$ on $\mathcal{M}(X)$, we need to establish some topology on $\mathcal{M}(X)$. At the beginning of this section we merely assume the ambient space $X$ is a topological space with the Borel  $\sigma$-algebra $\mathscr{A}$, in general. In order to get some particular results we will set restrictions on the ambient space $X$ in due course.  We discuss the continuity under three popular topologies on $\mathcal{M}(X)$, these are, the weak, setwise  and TV topology. These notions presumably differ from one another in strength, more desirable properties are possible in case one considers problems on some higher-strength topology. While some people prefer to work under the weak topology on  $\mathcal{M}(X)$, the latter two topologies are in fact widely used in probability theory.

\begin{definition}
For a sequence of measures $\{\nu_n\in\mathcal{M}(X)\}_{n=1}^\infty$, we say $\{\nu_n\}_{n=1}^\infty$ converges \emph{weakly (narrowly)} to $\nu\in\mathcal{M}(X)$, if 
\begin{center}
$\lim_{n\rightarrow\infty }\int_X f(x) d\nu_n=\int_X f(x) d\nu$
\end{center}
for any  bounded continuous real-valued function $f$ on $X$.
\end{definition}
One is also recommended to \cite{Bil1, Bil2, Kle, Mat} for more on the notion. Denote the convergence in this sense by 
\begin{center}
$\nu_n\stackrel{w}{\rightarrow}\nu$ 
\end{center}
as $n\rightarrow\infty$.

\begin{definition}
 A sequence of measures $\{\nu_n\in\mathcal{M}(X)\}_{n=1}^\infty$ is said to converge \emph{setwisely} to $\nu\in\mathcal{M}(X)$, if 
\begin{center}
$\lim_{n\rightarrow\infty }\nu_n(A)=\nu(A)$
\end{center}
for any $A\in\mathcal{A}$. 
\end{definition}

An equivalent way of defining the setwise convergence of measures stems from treating the measures as a  functional on the space of bounded Borel-measurable functions on $X$, similar to the way on defining weak convergence. 

\begin{definition}
A sequence of measures $\{\nu_n\in\mathcal{M}(X)\}_{n=1}^\infty$ converges setwisely to $\nu\in\mathcal{M}(X)$ if and only if
\begin{center}
$\lim_{n\rightarrow\infty }\int_X f(x) d\nu_n=\int_X f(x) d\nu$
\end{center}
for any bounded Borel-measurable function $f$ on $X$. 
\end{definition}
This is because the simple functions are dense among the bounded Borel-measurable functions on $X$ with the supremum norm. See for example \cite{Doo, FKZ, GR, HL, Las, LY}. Denote the convergence in this sense by 
\begin{center}
$\nu_n\stackrel{s}{\rightarrow}\nu$ 
\end{center}
as $n\rightarrow\infty$. 

Generally, we can not guarantee setwise convergence is stronger than  weak convergence, without excluding some exotic sequences of measures $\{\nu_n\in\mathcal{M}(X)\}_{n=1}^\infty$ or  some pathological ambient space $X$. The setwise convergence implies weak convergence if  the ambient space $X$ is metrizable, according to \cite[Theorem 2.1 (v)]{Bil1}. This is our setting of interest in all the subsequent sections. As a more demanding property of a sequence of measures,  it also guarantees more attractive results than weak convergence, for example, as one will see in the following, on the semi-continuity of measure-dimension mappings $dim_H^0$ and $dim_H^1$.  The converse is usually not true, as one can see from our Example \ref{exm1} and \ref{exm4}.  

\begin{theorem}\label{thm1}
Let $X$ be a metric space. The measure-dimension mapping $dim_H^1$ is lower semi-continuous under the setwise topology on $\mathcal{M}(X)$, that is, if $\nu_n\stackrel{s}{\rightarrow}\nu$ in $\mathcal{M}(X)$ as $n\rightarrow\infty$, then 
\begin{equation}\label{eq2}
\liminf_{n\rightarrow\infty} dim_H^1 \nu_n\geq dim_H^1 \nu.
\end{equation}
\end{theorem}
Compare the result with \cite[Theorem 5.2, Lemma 5.3]{RU2}. Alternatively, we have  upper semi-continuity for the  measure-dimension mapping $dim^0_H$, which shows the two measure-dimension mappings $dim^1_H$ and $dim^0_H$  are dual to each other in some sense. 

\begin{theorem}\label{thm4}
Let $X$ be a metric space. The measure-dimension mapping $dim^0_H$ is upper semi-continuous under the setwise topology on $\mathcal{M}(X)$, that is, if $\nu_n\stackrel{s}{\rightarrow}\nu$ in $\mathcal{M}(X)$ as $n\rightarrow\infty$, then 
\begin{equation}
\limsup_{n\rightarrow\infty} dim^0_H \nu_n\leq dim^0_H \nu.
\end{equation}
\end{theorem}

\begin{rem}
We would like to remind the readers about the semi-continuity result  \cite[Theorem 1.8]{HS} of M. Hochman and P. Shmerkin. For the unbiased Bernoulli convolutions $\nu_{\lambda}$ with contraction ratio $\lambda\in(0,1)$, it is well-known that $dim_H^0 \nu_{\lambda}=dim_H^1\nu_{\lambda}$ is not continuous at all  $\lambda\in(0,1)$. There are dimension-drops at the inverses of the Pisot numbers. Note that our Theorem \ref{thm4} together with  \cite[Theorem 1.8]{HS} does not contradict this fact as  $\nu_{\lambda}$ depends not continuously on the parameter $\lambda$ under the setwise topology.  
\end{rem}

Now we turn to the final concept of convergence of measures, it has much more simple description from the topological and metric point of view than the weak and setwise convergence.  

\begin{definition}
The \emph{total variation metric}  is defined by
\begin{center}
$\Vert \nu-\varrho\Vert_{TV}=\sup_{A\in\mathcal{A}}\{|\nu(A)-\varrho(A)|\}$
\end{center}   
between two measures $\nu, \varrho\in\mathcal{M}(X)$.
\end{definition}

Refer to \cite{Doo, FKZ, HL, Las, PS}). The corresponding induced topology is obviously courser than the setwise topology. Denote by 
\begin{center}
$\nu_n\stackrel{TV}{\rightarrow}\nu$ 
\end{center}
as $n\rightarrow\infty$ for sequences converging in this sense. Sometimes one can only get setwise convergence instead of TV convergence for sequences of measures. This happens just on some of our interested sequences of measures, see Lemma \ref{lem2} and Remark \ref{rem2}.

Note that there are no dynamical structures on the ambient space $X$ in all the above discussions in this section. We will consider cases when $X$ is equipped with CIFS and CGDMS structure on it in the following.

Now we give a quick description on the background of conformal iterated function systems introduced by R. D. Mauldin and M. Urba\'nski in the last 90s. Most of the notions cited here can be found in \cite{MU1, MU2}, so we do not referee everything individually any more. The ambient space $X\subset\mathbb{R}^d$ is supposed to be a compact connected set endowed with a metric $\rho$ in the following. 

\begin{definition}
For $\theta\in (0,1]$, an \emph{iterated function system} (IFS) is a collection of $C^{1+\theta}$ diffeomorphisms
\begin{center}
$S=\{s_i: X\rightarrow X\}_{i\in I}$
\end{center}
with a countable index set $I$  (with at least two elements).
\end{definition}
We assume here there exists $0<\gamma<1$, such that
\begin{equation}\label{eq1}
\rho(s_i(x),s_i(y))\leq \gamma\rho(x,y)
\end{equation}
for all $i\in I$ and any points $x, y\in X$. For $I^*=\cup_{n\geq 1}I^n$, let 
\begin{center}
$\omega=\omega_1\omega_2\cdots\in I^*\cup I^\infty$ 
\end{center}
be a finite or infinite word. Let $|\omega|$ be the length of the word, note that $|\omega|=\infty$ in case of $\omega\in I^\infty$. Define the \emph{shift map} 
\begin{center}
$\sigma: I^*\cup I^\infty\rightarrow I^*\cup I^\infty$ 
\end{center}
to be
\begin{center}
$\sigma(\omega)=\sigma(\omega_1\omega_2\cdots)=\omega_2\omega_3\cdots$.
\end{center}
For $\omega\in I^*\cup I^\infty$, let
\begin{center}
$s_\omega=s_{\omega_1}\circ s_{\omega_2}\circ\cdots$.
\end{center}
For $\omega\in I^*$ or $\omega\in I^\infty$, let $\omega|_n=\omega_1\omega_2\cdots\omega_n$ for $1\leq n\leq |\omega|$ or $1\leq n<\infty$. By (\ref{eq1}), the intersection $\cap_{n=1}^\infty s_{\omega|_n}(X)$ is a single point, so we define the \emph{projection map} 
\begin{center}
$\pi: I^\infty\rightarrow X$ 
\end{center}
to be
\begin{center}
$\pi(\omega)=\cap_{n=1}^\infty s_{\omega|_n}(X).$
\end{center}
The \emph{limit set} (or \emph{attractor}) of the IFS $S$ is defined as 
\begin{center}
$J=\pi(I^\infty)=\cup_{\omega\in I^\infty}\cap_{n=1}^\infty s_{\omega|_n}(X).$
\end{center}

For $\omega\in I^*$, let  
\begin{center}
$[\omega]=\{\omega'\in I^\infty: \omega'|_{|\omega|}=\omega\}$
\end{center}
be a \emph{cylinder set} in $I^\infty$. Let $\mathcal{B}_{I^\infty}$ be the  $\sigma$-algebra generated by the cylinder sets $\{[\omega]: \omega\in I^*\}$.  Let $\mathcal{B}_J$ be the pushing-forward  $\sigma$-algebra of $\mathcal{B}_{I^\infty}$ on $J$ under the projection map $\pi$.

The notion of conformal iterated function systems provides a general environment for many of the attractive properties to hold in its framework.  Let $X^o$ be the interior of $X$ in $\mathbb{R}^d$ and $|s_\omega'(x)|$ be the absolute value of the derivative at $x$. 
 
\begin{definition}\label{def7} 
A \emph{conformal iterated function system} 
\begin{center}
$S=\{s_i: X\rightarrow X\}_{i\in I}$ 
\end{center} 
is an IFS satisfying the following properties.
 
\begin{enumerate}[(1).]
\item $s_i(X^o)\subset X^o$ for any $i\in I$ while $s_{i_1}(X^o)\cap s_{i_2}(X^o)=\emptyset$ for any distinct pair of indexes $i_1,i_2\in I$.
\item For any $x\in\partial X$, there exists an open cone with vertex $x$ of small size in $X^o$.
\item Every map $s_i: X\rightarrow X$ has a $C^{1+\theta}$ extension to an open connected set $W\supset X$ in $\mathbb{R}^d$.
\item For any word $\omega\in I^*$ and any distinct pair of points $x, y\in W$, there exists $K\geq 1$ such that
\begin{center}
$|s_\omega'(x)|\leq K|s_\omega'(y)|$.
\end{center}

\end{enumerate} 
\end{definition}

The first property is a strong kind of \emph{open set condition} while the last one is  usually coined  the \emph{bounded distortion property} (BDP).

We focus on the probability measures supported on the limit set $J$ of a CIFS in the following (or their trivial extensions to the whole ambient set $X$). An important tool in CIFS is a pre-mature measure called $t$-conformal measure on $J$. Its development in the theory of dynamical systems follows the sequential work of S. Patterson \cite{Pat}, D. Sullivan \cite{Sul1, Sul2}, R. D. Mauldin and M. Urba\'nski \cite{MU1, MU3}.  

\begin{definition}
For $t\geq 0$, a probability measure $m\in \mathcal{M}(J)$ is said to be $t$-\emph{conformal} if 
\begin{equation}\label{eq4}
m(s_i(A))=\int_A |s'_i|^tdm
\end{equation}  
for any $i\in I$ and 
\begin{equation}\label{eq18}
m(s_{i_1}(X)\cap s_{i_2}(X))=0
\end{equation}
for any $i_1,i_2\in I$ with $i_1\neq i_2$. 
\end{definition}

It does not always exist for any CIFS (see the condition Theorem \ref{thm20} (d) for its existence in Section \ref{sec6}). In case of its existence, according to  Theorem \ref{thm20} (c), there is an unique ergodic measure $m^*$ on $J$ equivalent to it.

The \emph{topological pressure} of the CIFS $S$ is defined as the following,
\begin{equation}\label{eq49}
P(t)=\lim_{n\rightarrow\infty} \frac{1}{n} \log \psi_n(t)=\lim_{n\rightarrow\infty} \frac{1}{n} \log \Sigma_{\omega\in I^n} \| s'_{\omega}(x)\|^t.
\end{equation}

\begin{definition}\label{def2}
A CIFS $S$ is called \emph{regular} if $P(t)=0$ admits a solution, or else it is called \emph{irregular}.
\end{definition}

For a measure $\nu\in \mathcal{M}(I^\infty)$, let $h_\nu(\sigma)$ be its measure-theoretic entropy with respect to the shift map $\sigma$. Let 
\begin{center}
$\lambda_{\nu}(\sigma)=\int_{I^\infty} -\log|s_{\omega_1}'(\pi(\sigma(\omega)))|d\nu$
\end{center}
be its average Lyapunov exponent. Let $\mathbb{N}_n=\{1, 2, \cdots, n\}$ be the $n$-th truncation for $n\in\mathbb{N}$. Our first main result essentially deals with Hausdorff dimensions of exploding conformal measures $m$ and their equivalent ergodic evolutions $m^*$ originated from CIFS as following. This theorem gives a positive answer to \cite[Problem 7.3]{MU1}. 

\begin{theorem}\label{thm3}
Suppose  
\begin{center}
$S=\{s_i: X\rightarrow X\}_{i\in \mathbb{N}}$ 
\end{center}
is a regular CIFS  with infinite index set $\mathbb{N}$. Let $\mu_n^*$ be the corresponding ergodic measure on $\mathbb{N}_n^\infty$. Then we have 
\begin{equation}\label{eq42}
dim_H^0 m =dim_H^0 m^*=dim_H^1 m =dim_H^1 m^*=HD(J)=\lim_{n\rightarrow\infty}\cfrac{h_{\mu_n^*}(\sigma)}{\lambda_{\mu_n^*}(\sigma)}.
\end{equation}

\end{theorem}

Now we turn to the graph directed Markov systems (GDMS). Our main source of reference comes from \cite{CLU, CTU, MU3}, while one is also strongly recommended to refer to \cite{Atn, EM, KK1, KK2, KSS, MW, PU, Roy} for more historical or most recent development of the research topics on GDMS. Most of the following definitions can be found in \cite{CLU, CTU, MU3}, so we do not pinpoint every individual one any more.  The construction  of GDMS extracts the quintessence from a broad family of important systems, for example, the Kleinian and complex hyperbolic Schottky groups, Apollonian circle packings and CIFS. When considering it as an extensive class of CIFS, one can define it by letting the domain (range)-space of the maps in a CIFS varying while narrowing the permissive infinite words  built from a given index set, according to ones' need.

\begin{definition}
A \emph{graph directed Markov system}
\begin{center}
$S=\big\{V, E, A, t, i, \{X_v\}_{v\in V}, \{s_e\}_{e\in E}, \gamma\big\}$
\end{center} 
is composed of the following items.

\begin{itemize}

\item $V$ is a finite index set whose elements are called \emph{vertices}.

\item Associated to every $v\in V$ there is a non-empty compact metric space $X_v$.

\item Let $E$ be a countable set of \emph{directed edges}. Every edge $e\in E$ is associated with two vertices in $V$, with the initial one denoted by $i(e)$ and the terminal one denoted by $t(e)$. 

\item Associated to every $e\in E$ there is an injective map 
\begin{center}
$s_e: X_{t(e)}\rightarrow X_{i(e)}$,
\end{center}
while the whole family $\{s_e\}_{e\in E}$ is uniformly Lipschitz with a uniform Lipschitz constant $0<\gamma<1$.

\item We build words using elements in $E$ according to the \emph{incidence matrix} 
\begin{center}
$A: E\times E\rightarrow \{0,1\}$,
\end{center}
that is, an edge $e'\in E$ is allowed to follow another edge $e\in E$ if and only if
\begin{center}
$A_{ee'}=A(e, e')=1$.  
\end{center}
To make composition of maps associated with adjacent edges $e, e'\in E$ in a word, we require that if  $A_{ee'}=A(e, e')=1$, then  
\begin{center}
$t(e)=i(e')$.
\end{center}

\item We assume the whole graph has some connectivity. For any vertex  $v\in V$, we require there exist at least two edges $e, e'\in E$ with 
\begin{center}
$t(e)=i(e')=v$.
\end{center}

\end{itemize}

\end{definition}

\begin{rem}
The set $E$ is possible to be infinite, so possibly there are many edges sharing the same initial and terminal vertices. These edges are still regarded as different ones, distinguished from each other by their associated injective mappings. 
\end{rem}

We only allow the $A$-\emph{admissible} (finite or infinite) words in the  set $E^{*,A}\cup E^{\infty,A}$, which is defined as
\begin{center}
$E^{*,A}=\cup_{n\in\mathbb{N}}E^{n,A}=\cup_{n\in\mathbb{N}}\{\omega\in E^n: A_{\omega_i\omega_{i+1}}=1 \mbox{\ for any\ } 1\leq i\leq n-1\}$
\end{center}
and
\begin{center}
$E^{\infty,A}=\{\omega\in E^\infty: A_{\omega_i\omega_{i+1}}=1 \mbox{\ for any\ } 1\leq i<\infty\}$.
\end{center}
For the symbolic dynamics on  $E^{\infty,A}$ see \cite[Chapter 2]{MU3} or \cite[Section 2]{CLU}.

All the notations for words, cylinder sets (with all finite and infinite words being admissible) and shift map are applied here. Now we abuse notations from CIFS, but the readers should keep in mind that these are more general objects than those for CIFS. For two words $\omega, \omega'\in E^{*,A}\cup E^{\infty,A}$, let $\omega\wedge \omega'$ be the longest common block starting from their first edges.  For a word $\omega\in E^{*,A}$, let
\begin{center}
$i(\omega)=i(\omega_1)$ and $t(\omega)=t(\omega_{|\omega|})$.
\end{center}

We emphasize those GDMS with some 'core' in the sense of connectivity in our consideration. A  GDMS (or its incidence matrix $A$) is called \emph{irreducible} if there exists a set 
\begin{center}
$\Lambda\subset E^{*,A}$, 
\end{center}
such that  for any two edges $e, e'\in E$, there exists $\omega\in\Lambda$ with 
\begin{center}
$e\omega e'\in E^{*,A}$. 
\end{center}
It is called  \emph{primitive} if the words in $\Lambda$ are all of the same length. It is called 
\emph{finitely irreducible} if the above set $\Lambda$ with the corresponding property is finite, and  \emph{finitely primitive} if the above set $\Lambda$ with the corresponding property is finite and contains all words of the same length.

For a word $\omega\in E^{\infty,A}$, since the family $\{s_e\}_{e\in E}$ is uniformly Lipschitz, the intersection
\begin{center}
$\cap_{n\in\mathbb{N}}s_{\omega|_n}(X_{t(\omega_n)})$
\end{center} 
is a singleton. Let   
\begin{center}
$X=\oplus_{v\in V}X_v$
\end{center}
be the disjoint union the the sets $\{X_v\}_{v\in V}$. Then we can define the projection map 
\begin{center}
$\pi: E^{\infty,A}\rightarrow\oplus_{v\in V}X_v$ 
\end{center}
as
\begin{center}
$\pi(\omega)=\cap_{n\in\mathbb{N}}s_{\omega|_n}(X_{t(\omega_n)})$.
\end{center}  
The limit set is now defined as 
\begin{center}
$J=\pi(E^{\infty,A})$.
\end{center}

Our main interest in this work would be the conformal graph directed Markov systems (CGDMS). 

\begin{definition}\label{def9}
A \emph{conformal graph directed Markov system} is a GDMS
\begin{center}
$S=\big\{V, E, A, t, i, \{X_v\}_{v\in V}, \{s_e\}_{e\in E}, \gamma\big\}$
\end{center} 
satisfying the following properties:

\begin{enumerate}[(1).]
\item $X_v$ is a compact connected subset of the Euclidean space $\mathbb{R}^d$  for any $v\in V$, with $X_v=\overline{X_v^o}$.   
\item For any two distinct edges  $e, e'\in E$, we have 
\begin{center}
$s_{e}(X_{t(e)}^o)\cap s_{e'}(X_{t(e')}^o)=\emptyset$.
\end{center}
\item For every $v\in V$, there is an open connected set $W_v\supset X_v$, such that for any $e\in E$, the map $s_e$ extends to a \emph{conformal map} from $W_{t(e)}$ to  $W_{i(e)}$.
\end{enumerate} 

\end{definition}

\begin{rem}
We choose our ambient space to be in $\mathbb{R}^d$ here, while most results can probably be extended to some  general metric spaces, with suitable assumptions on the boundary of the sets $\{X_v\}_{v\in V}$ and distortions of the family of maps $\{s_e\}_{e\in E}$.
\end{rem}

There should be some explanations on the \emph{conformal map} in Definition \ref{def9} (3) above. In case of $d=1$, the conformal map is assumed to be a monotone $C^1$ diffeomorphism. In case of $d=2$, it means a holomorphic or anti-holomorphic map on the complex plane $\mathbb{C}$. In case of $d\geq 3$, it means a M\"obius transformation. If a GDMS satisfies Definition \ref{def9} (1) and (3), Mauldin and Urba\'nski demonstrated that (\cite[P73,(4f)]{MU3}) it also must satisfy the following BDP property,
\begin{enumerate}[(4).]
\item There exists a constant $K\geq1$, such that 
\begin{center}
$|s_\omega'(x)|\leq K |s_\omega'(y)|$
\end{center}
for any $\omega\in E^{*,A}$ and any $x,y\in X_{t(\omega)}$. 
\end{enumerate} 

One can also work by simply assuming the family of maps $\{s_e\}_{e\in E}$ to be $C^1$ with property (4) instead of conformal. For CGDMS on the ambient space of \emph{Carnot groups} equipped with a sub-Riemannian metric, see \cite{CTU}.

For an infinite CGDMS 
\begin{center}
$S=\big\{V, E, A, t, i, \{X_v\}_{v\in V}, \{s_e\}_{e\in E}, \gamma\big\}$,
\end{center}
a probability measure $\mathbf{m}$ supported on $J$ is said to be $t$-\emph{conformal} for some $t\geq 0$, if it satisfies
\begin{equation}
\mathbf{m}(s_\omega(A))=\int_A |s'_\omega|^td\mathbf{m}
\end{equation}  
for any $\omega\in E^{*,A}$ and 
\begin{equation}
\mathbf{m}\big(s_\omega(X_{t(\omega)})\cap s_{\omega'}(X_{t(\omega')})\big)=0
\end{equation}
for any two incomparable words $\omega,\omega'\in E^{*,A}$. Let $\mathbf{m}^*$ be the equivalent ergodic measures to the conformal one $\mathbf{m}$. Our second main result deals with Hausdorff dimensions of exploding conformal measures $\mathbf{m}$ and their equivalent ergodic evolutions $\mathbf{m}^*$ originated from CGDMS as following.

\begin{theorem}\label{cor4}
For an infinite finitely primitive regular CGDMS, we have 
\begin{center}
$dim_H^0 \mathbf{m}=dim_H^1 \mathbf{m}=dim_H^0 \mathbf{m}^*=dim_H^1 \mathbf{m}^*=HD(J)=\lim_{n\rightarrow\infty}\cfrac{h_{\bm \mu_n^*(\sigma)}}{\lambda_{\bm \mu_n^*(\sigma)}}$,
\end{center}
in which $\bm \mu_n^*$ is the corresponding ergodic measure on the corresponding truncated symbolic space for $n\in\mathbb{N}$ large enough.
\end{theorem}

Compare the result with \cite[Corollary 4.4.6]{MU3}.

\section{Discontinuity of the measure-dimension mappings under  weak, setwise, TV  topology and the proofs of continuity results}\label{sec7}

In this section, without special declaration, $X$ is supposed to be a metric space endowed with a metric $\rho$, which is necessary for us to define the Hausdorff dimensions of sets in $X$. First, to provide a general intuition on the continuity of the measure-dimension mappings in our consideration, we present the readers with the following result.
\begin{theorem}\label{thm18}
The measure-dimension mappings $dim_H^0, dim^1_H, dim_C$ and $dim_{MC}$ from $\mathcal{M}(X)$ to $[0,+\infty)$ are not continuous under the weak, setwise or TV topology.
\end{theorem}
\begin{proof}
See for example our Example \ref{exm6} and \ref{exm5}.
\end{proof}
It would be interesting to ask whether some measure-dimension mappings can be defined to be equipped with continuity under some topology, possibly at the cost of losing some known properties of some popular measure-dimension mappings.

While the weak convergence is used frequently in analysing the space $\mathcal{M}(X)$ with IFS structure on $X$, unfortunately, even semi-continuity of  the measure-dimension mappings can not be guaranteed in this sense, see  Remark \ref{rem1}. We have to resort to the stronger notion of  setwise  convergence on $\mathcal{M}(X)$ to achieve our aims.

Unfortunately, the two measure-dimension mappings $dim_C$ and $dim_{MC}$ do not have any semi-continuity under the strongest topology in general. 

\begin{theorem}\label{thm5}
The two measure-dimension mappings $dim_C$ and $dim_{MC}$ from $\mathcal{M}(X)$ to $[0,+\infty)$  are neither upper semi-continuous nor lower semi-continuous under the TV topology.  
\end{theorem}
\begin{proof}
See our Example \ref{exm5} and \ref{exm7}.
\end{proof}

However, in virtue of \cite[Theorem 2.6]{MMR}, if some restrictions are set on the sequence of the measures $\{\nu_n\}_{n=1}^\infty$, we have some partial semi-continuity results.
\begin{corollary}
For a sequence of  measures $\{\nu_n\in\mathcal{M}(X)\}_{n=1}^\infty$, if  $\nu_n$ has an essentially bounded density with respect to $\nu$ (which implies that $\nu_n$ is absolutely continuous with respect to $\nu$) for any $n\in\mathbb{N}$, then 
\begin{center}
$\liminf_{n\rightarrow\infty} dim_C \nu_n\geq dim_C \nu$.
\end{center} 
\end{corollary}

The restriction on the sequence of the measures $\{\nu_n\}_{n=1}^\infty$ can be relaxed, by \cite[Theorem 2.10]{MMR}, to give the following lower semi-continuity result for $dim_{MC}$ in due course. 
\begin{corollary}
For any sequence of  measures $\{\nu_n\in\mathcal{M}(X)\}_{n=1}^\infty$, if  $\nu_n$ is absolutely continuous with respect to $\nu$ for any $n\in\mathbb{N}$, then 
\begin{center}
$\liminf_{n\rightarrow\infty} dim_{MC} \nu_n\geq dim_{MC} \nu$.
\end{center} 
\end{corollary}

We first give an example to show that for a weakly convergent sequence $\nu_n\stackrel{w}{\rightarrow}\nu$ as $n\rightarrow\infty$, any of the four sequences 
\begin{center}
$\{dim_H^0 \nu_n\}_{n=1}^\infty, \{dim^1_H \nu_n\}_{n=1}^\infty, \{dim_C \nu_n\}_{n=1}^\infty, \{dim_{MC} \nu_n\}_{n=1}^\infty$ 
\end{center}
may not converge. In this and all the examples below, we set $X\subset \mathbb{R}^d$ endowed with the Euclidean metric $\rho_u$ for some positive integer $d$. Let $\mathfrak{L}^d$ be the $d$-dimensional Lebesgue measure on $\mathbb{R}^d$.

\begin{exm}\label{exm1} 
Define a sequence of probability measures $\{\nu_n\}_{n\in\mathbb{N}}$ on $[0,1]$ to be 
\begin{center}
$ \nu_n=\left\{
\begin{array}{ll}
n\mathfrak{L}^1|_{[0,\frac{1}{n}]} & \mbox{ if } n \mbox{ is } odd,\\
\delta_{\frac{1}{n}} & \mbox{ if } n \mbox{ is } even.\\
\end{array}
\right.$ 
\end{center} 
in which $\delta_{\frac{1}{n}}$ is the Dirac measure at the point $\frac{1}{n}$. Let $\nu=\delta_{0}$ be the Dirac measure at $0$.
\end{exm} 

Obviously we have $\nu_n\stackrel{w}{\rightarrow}\nu$ as $n\rightarrow\infty$ (but not setwisely). By  \cite[Proposition 2.1]{You} and \cite[Lemma 2.8]{MMR}, direct computations give that 
\begin{center}
$dim_H^0 \nu_n=dim^1_H \nu_n=dim_C \nu_n=dim_{MC} \nu_n=\left\{
\begin{array}{ll}
1 & \mbox{ if } n \mbox{ is } odd,\\
0 & \mbox{ if } n \mbox{ is } even.\\
\end{array}
\right.$ 
\end{center}
So any of the four sequences 
\begin{center}
$\{dim_H^0 \nu_n\}_{n=1}^\infty, \{dim^1_H \nu_n\}_{n=1}^\infty, \{dim_C \nu_n\}_{n=1}^\infty$ and $\{dim_{MC} \nu_n\}_{n=1}^\infty$ 
\end{center}
does not converge in Example \ref{exm1}.

We then give an example to show that in case of $\nu_n\stackrel{w}{\rightarrow}\nu$, even if all the four sequences 
\begin{center}
$\{dim_H^0 \nu_n\}_{n=1}^\infty, \{dim^1_H \nu_n\}_{n=1}^\infty, \{dim_C \nu_n\}_{n=1}^\infty$ and $\{dim_{MC} \nu_n\}_{n=1}^\infty$ 
\end{center}
converge, their limits may not equal 
\begin{center}
$dim_H^0 \nu, dim^1_H \nu, dim_C \nu$ and $dim_{MC} \nu$
\end{center}
respectively. The example is  borrowed from \cite[Example 2.2]{Bil1} essentially.

\begin{exm}[Billingsley]\label{exm4}
Define a sequence of probability measures $\nu_n$ on $[0,1]$ to be  
\begin{center}
$ \nu_n=\frac{1}{n}\Sigma_{i=1}^n\delta_{\frac{i}{n}}$ 
\end{center}
for $n\in\mathbb{N}$.
\end{exm} 
Obviously $ \nu_n\stackrel{w}{\rightarrow}\mathfrak{L}^1|_{[0,1]}$ as $n\rightarrow\infty$ (but not setwisely). However, we have
\begin{center}
$\lim_{n\rightarrow\infty}dim_H^0 \nu_n=\lim_{n\rightarrow\infty} dim^1_H \nu_n=\lim_{n\rightarrow\infty} dim_C \nu_n=\lim_{n\rightarrow\infty} dim_{MC} \nu_n=0$,
\end{center}
while
\begin{center}
$dim_H^0 \mathfrak{L}^1|_{[0,1]}=dim^1_H \mathfrak{L}^1|_{[0,1]}=dim_C \mathfrak{L}^1|_{[0,1]}=dim_{MC} \mathfrak{L}^1|_{[0,1]}=1$.
\end{center}

One might think that things will be better if the atomic measures are excluded, however, the following example shows that the discontinuity still exists among sequences of non-atomic measures.

\begin{exm}\label{exm3}
Suppose $k\geq 2$ is an integer. Let
\begin{center}
$S=\{s_i: [0,1]\rightarrow[0,1]\}_{i=1}^k$ 
\end{center}
be an affine IFS with the contraction ratio $0<|s_i|:=|s_i'(x)|<1$ being a fixed number for any $1\leq i\leq k$ on $X=[0,1]$. We require it satisfies the \emph{strong separation condition}:
\begin{center}
$s_i([0,1])\cap s_j([0,1])=\emptyset$
\end{center}
for any $1\leq i\neq j\leq k$. We do not specify the terminals of the the intervals $\{s_i([0,1])\}_{i=1}^k$ as they have no effect on the properties we would like to demonstrate through the example. Let $0<h<1$ be the unique solution of the \emph{Bowen equation}
\begin{center}
$\sum_{i=1}^k|s_i|^h=1$.
\end{center}
According for example to \cite[Lemma 3.14, Theorem 3.18]{MU1}, there is a unique $h$-conformal and ergodic measure (with respect to the push-forward of the shift map under the projection) $\nu$ on the limit set $J$ in this case. 

Let 
\begin{center}
$\mathbb{N}_k:=\{1,2,\cdots,k\}$ 
\end{center}
 be the $k$-truncation of $\mathbb{N}$. For every fixed $n\in\mathbb{N}$, let
\begin{center}
$\mathbb{N}_k^n=\{\omega: \omega=\omega_1\omega_2\cdots\omega_n, \omega_i\in\mathbb{N}_k \mbox{\ for any } 1\leq i\leq n\}$
\end{center}     
be the collection of length-$n$ concatenation-words of $\mathbb{N}_k$. Let 
\begin{center}
$X_\omega=s_\omega([0,1])=s_{\omega_1}\circ s_{\omega_2}\cdots \circ s_{\omega_n}([0,1])$
\end{center}
for any $\omega\in\mathbb{N}_k^n$ and any $n\in\mathbb{N}$. Define a sequence of probability measures $\{\nu_n\}_{n\in\mathbb{N}}$ to be
\begin{center}
$\nu_n=\sum_{\omega\in\mathbb{N}_k^n}|s_\omega|^h\mathfrak{L}^1|_{X_\omega}=\sum_{\omega\in\mathbb{N}_k^n}(|s_{\omega_1}||s_{\omega_2}|\cdots|s_{\omega_n}|)^h\mathfrak{L}^1|_{X_\omega}$
\end{center} 
on $[0,1]$ for any $n\in\mathbb{N}$. It is supported on $\cup_{\omega\in\mathbb{N}_k^n} X_\omega$ for each $n\in\mathbb{N}$ obviously. In fact $\nu_n$ is the $n$-th canonical mass distribution of the linear cookie-cutter map in each cutting step. 
\end{exm}

Note that none of the measures above is atomic. It is easy to show that for any open set $L\subset [0,1]$, we have $ \nu_n(L)\geq \nu(L)$. So by  \cite[Theorem 2.1]{Bil1}, we have 
\begin{center}
$\nu_n\stackrel{w}{\rightarrow}\nu$ 
\end{center}
as $n\rightarrow\infty$ (but not setwisely). By \cite[Theorem 2.6, Theorem 2.10]{MMR} and our Theorem \ref{thm3},  we have 
\begin{center}
$\lim_{n\rightarrow\infty}dim_H^0 \nu_n=\lim_{n\rightarrow\infty} dim^1_H \nu_n=\lim_{n\rightarrow\infty} dim_C \nu_n=\lim_{n\rightarrow\infty} dim_{MC} \nu_n=1$,
\end{center}
while
\begin{center}
$dim_H^0 \nu=dim^1_H \nu=dim_C \nu=dim_{MC} \nu=\frac{\log 2}{\log 3}$.
\end{center}

One might also hope that escalating the strength of convergence of measures may save the continuity, but we will give an example to show that  convergence of dimensions is not true even for TV convergent sequence of measures (see also Example \ref{exm5}).

 \begin{exm}\label{exm6}
Define a sequence of probability measures $\nu_n$ on $[0,2]$ to be 
\begin{center}
$ \nu_n=\frac{n-1}{n} \delta_{0}+\frac{1}{n} \mathfrak{L}^1|_{[1,2]}.$ 
\end{center} 
\end{exm}
It is easy to see that $\| \nu_n-\delta_{0} \|_{TV}=\frac{1}{n}\rightarrow 0$ as $n\rightarrow\infty$, so $ \nu_n\stackrel{TV}{\rightarrow}\delta_{0}$, however,

\begin{center}
$\lim_{n\rightarrow\infty}dim_H^1 \nu_n=1>0=dim_H^1 \delta_{0}=dim^0_H \delta_{0}$.
\end{center}

In fact, continuity of measure-dimension mappings $dim_H^0$ and $dim^1_H$  has no chance to be true if there is no dimensional restrictions on the sequences $\{dim_H^0 \nu_n\}_{n=1}^\infty$ and $\{dim^1_H \nu_n\}_{n=1}^\infty$. Luckily we do have some semi-continuity property for $dim^1_H$ under the setwise topology. Now we prove Theorem \ref{thm1}.\\

Proof of Theorem \ref{thm1}:

\begin{proof}
Assume that (\ref{eq2}) does not hold, that is,  $\liminf_{n\rightarrow\infty} dim_H^1 \nu_n< dim_H^1 \nu$. Then we can find a sequence of positive integers $\{n_k\}_{k=1}^\infty$ and a real number $a<dim_H^1 \nu$, such that
\begin{center}
$\lim_{k\rightarrow\infty} dim_H^1 \nu_{n_k}=a$.
\end{center} 
So there exists a sequence of measurable sets $\{A_k\}_{k=1}^\infty$, such that 
\begin{center}
$\nu_{n_k}(A_k)=1$ and $HD(A_k)<a+\varepsilon< dim_H^1 \nu$
\end{center}
for some small $\varepsilon>0$. Now let $A=\cup_{k=1}^\infty A_k$. Since $ \nu_n\stackrel{s}{\rightarrow}\nu$, we have
\begin{center}
$\nu(A)=\lim_{k\rightarrow\infty}\nu_{n_k}(A)=1.$
\end{center}
However, $HD(A)=\sup_k\{HD(A_k)\}<a+\varepsilon<dim_H^1 \nu$. This contradicts the definition of $dim_H^1 \nu$, which justifies our theorem.
\end{proof}

\begin{rem}\label{rem1}
Theorem \ref{thm1} is obviously wrong for weakly convergent sequences $\nu_n\stackrel{w}{\rightarrow}\nu$ as $n\rightarrow\infty$ in $\mathcal{M}(X)$, as one can see from our Example \ref{exm4}. Together with Example \ref{exm3}, one can see that both lower semi-continuity and upper semi-continuity are not true for the two measure-dimension mappings $dim_H^0$ and $dim^1_H$  under the weak topology. 

\end{rem}

\begin{corollary}\label{cor1}
If $ \nu_n\stackrel{s}{\rightarrow}\nu$ as $n\rightarrow\infty$ and $\esssup\{dim_H^1 \nu_n\}_{n=1}^\infty\leq dim_H^1 \nu$, then 
\begin{center}
$\lim_{n\rightarrow\infty} dim_H^1 \nu_n=dim_H^1 \nu$.
\end{center}
\end{corollary}

\begin{proof}
This follows directly from Theorem \ref{thm1}.
\end{proof}

Note that the proof of Theorem \ref{thm1} can not be applied to the measure-dimension mapping $dim^0_H$, because if we choose a sequence of sets $\{A_k\}_{k=1}^\infty$, such that 
\begin{center}
$\nu_{n_k}(A_k)>0$ and $dim_H^0 A_k<a+\varepsilon< dim_H^0 \nu$
\end{center}
we can not guarantee $\nu(A)=\lim_{k\rightarrow\infty}\nu_{n_k}(A)>0.$ In fact one can see from the following example that lower semi-continuity of the measure-dimension mapping $dim^0_H$ is not always true for sequences of measures $\nu_n\stackrel{s}{\rightarrow}\nu$ in $\mathcal{M}(X)$.

\begin{exm}\label{exm5}
Define a sequence of probability measures $\{\nu_n\}_{n\in\mathbb{N}}$ on $[0,1]$ to be 
\begin{center}
$ \nu_n=\frac{1}{n} \delta_{0}+\mathfrak{L}^1|_{[\frac{1}{n},1]}$ 
\end{center} 
for $n\in\mathbb{N}$.
\end{exm}
It is easy to see that $\| \nu_n-\mathfrak{L}^1|_{[0,1]} \|_{TV}=\frac{1}{n}\rightarrow 0$ as $n\rightarrow\infty$, so $ \nu_n\stackrel{TV}{\rightarrow}\delta_{0}$, however,

\begin{center}
$\lim_{n\rightarrow\infty}dim^0_H \nu_n=0<1=dim^0_H \mathfrak{L}^1|_{[0,1]}=dim_H^1 \mathfrak{L}^1|_{[0,1]}$.
\end{center}
Simple calculations also show that, in Example \ref{exm5},
\begin{center}
$dim_C \nu_n=dim_{MC} \nu_n=\lim_{r\rightarrow 0}\cfrac{\log \int\nu_n(B(x,r))d\nu_n}{\log r}=\cfrac{\log\big((\frac{1}{n})^2+2r\frac{n-1}{n}\big)}{\log r}=0$,
\end{center} 
while
\begin{center}
$dim_C \mathfrak{L}^1|_{[0,1]}=dim_{MC} \mathfrak{L}^1|_{[0,1]}=1$.
\end{center}
Thus the example also justifies the fact that the two measure-dimension mappings $dim_C$ and $dim_{MC}$ can not be lower semi-continuous under the TV topology. Considering Theorem \ref{thm5}, we still owe the readers an example of a convergent sequence of measures violating the upper semi-continuity of $dim_C$ and $dim_{MC}$ under the TV topology. The following example is an evolution of \cite[Example 2.5]{MMR}.

\begin{exm}\label{exm7}
Let $a\in(0,1)$ be a real number. Consider the sequence $\{a^{n^2}\}_{n=0}^\infty\subset [0,1]$. 
Define a sequence of probability measures $\nu_n$ on $[0,1]$ to be 
\begin{center}
$\nu_n=\cfrac{1-a}{1-a^{n+1}}\sum_{i=0}^n\cfrac{a^i}{a^{i^2}-a^{{(i+1)}^2}}\mathfrak{L}^1|_{[a^{i^2},a^{{(i+1)}^2}]}$
\end{center}
for any $n\in\mathbb{N}$. Let $\nu$ be the probability measure on $[0,1]$ with
\begin{center}
$\nu=(1-a)\sum_{i=0}^\infty \cfrac{a^i}{a^{i^2}-a^{{(i+1)}^2}}\mathfrak{L}^1|_{[a^{i^2},a^{{(i+1)}^2}]}$.
\end{center}

\end{exm}

First we show the sequence of measures converges under the TV topology.
\begin{lemma}\label{lem7}
For the sequence of measures $\{\nu_n\}_{n\in\mathbb{N}}$ in Example \ref{exm7}, we have
\begin{center}
$\nu_n\stackrel{TV}{\rightarrow}\nu$ as $n\rightarrow\infty$.
\end{center}
 \end{lemma}
 
\begin{proof}
In this case we have the following estimation on the total variation of the difference between the two measures $\nu_n$ and $\nu$,
\begin{center}
$\|\nu_n-\nu\|_{TV}=\max\{a^{n+1},\cfrac{a^{n+1}}{1-a^{n+1}}\}=\cfrac{a^{n+1}}{1-a^{n+1}}$
\end{center}
for any $n\in\mathbb{N}$. So $\lim_{n\rightarrow\infty}\|\nu_n-\nu\|_{TV}=0$, which justifies its convergence under the TV topology.
\end{proof} 

Although the correlation dimension and modified correlation dimension appear as global concepts, they are sometimes dominated by some properties on a subset of $X$ of relatively small measure, as indicated by the following result.

\begin{lemma}\label{lem6}
Let $\nu\in\mathcal{M}(X)$ be a probability measure. For any $r$ small enough, if there exists a subset $Y_r\subset X$ of positive measure, such that for \emph{a.e.} $x\in Y_r$, the following estimation holds,
\begin{equation}\label{eq28}
\nu(B(x,r))\nu(Y_r)\geq r^{o(1)},
\end{equation}
in which $o(1)$ is a positive term satisfying $\lim_{r\rightarrow 0} o(1)=0$. Then we have
\begin{center}
$dim_C\nu=dim_{MC}\nu=0$.
\end{center}
\end{lemma}
\begin{proof}
Under the above assumptions, we have
\begin{center}
$\int \nu(B(x,r))d\nu(x)\geq \nu(B(x,r))\nu(Y_r)$.
\end{center}
So 
\begin{center}
$\log \int \nu(B(x,r))d\nu(x)\geq\log \big(\nu(B(x,r))\nu(Y_r)\big)\geq o(1)\log r$,
\end{center}
which gives that
\begin{center}
$\cfrac{\log \int \nu(B(x,r))d\nu(x)}{\log r}\leq o(1)$.
\end{center}
The proof ends by letting $r\rightarrow 0$ in the above estimation.  
\end{proof}

Now we pay attention to the correlation and modified correlation dimensions of the limit measure of the sequence $\{\nu_n\}_{n\in\mathbb{N}}$ in Example \ref{exm7}.
\begin{proposition}[Mattila-Mor\'an-Rey]\label{pro1}
For the limit measure $\nu$ in Example \ref{exm7}, we have
\begin{center}
$dim_C\nu=dim_{MC}\nu=0$.
\end{center}
\end{proposition}
\begin{proof}
We achieve the conclusion by Lemma \ref{lem6}. Suppose $r$ is a number small enough such that
\begin{equation}\label{eq29}
a^{n_r^2}\leq r<a^{(n_r+1)^2}
\end{equation}
for some $n_r\in\mathbb{N}$. Let $Y_r=Y_{n_r}=[0,a^{n_r^2}]$, then we have
\begin{center}
$\nu(Y_{n_r})=(1-a)(a^{n_r}+a^{n_r+1}+\cdots)=(1-a)\cfrac{a^{n_r}}{(1-a)}=a^{n_r}$.
\end{center}
For any $x\in Y_{n_r}$, since $Y_{n_r}\subset B(x,r)$, we have
\begin{center}
$\nu(B(x,r))\geq \nu(Y_{n_r})=a^{n_r}$.
\end{center}
So 
\begin{center}
$\nu(B(x,r))\nu(Y_r)\geq a^{2n_r}$
\end{center}
for any $x\in Y_r$. Considering (\ref{eq29}), the condition (\ref{eq28}) is satisfied in this case. So by 
Lemma \ref{lem6}, we have
\begin{center}
$dim_C\nu=dim_{MC}\nu=0$.
\end{center}
\end{proof}

Obviously for the sequence of measures $\{\nu_n\}_{n\in\mathbb{N}}$ in Example \ref{exm7}, we have
\begin{center}
$dim_C\nu_n=dim_{MC}\nu_n=1$.
\end{center}
Combining this with Lemma \ref{lem7} and Proposition \ref{pro1}, we conclude that the measure dimension mappings $dim_C$ and $dim_{MC}$ can not be upper semi-continuous in general.

Now we prove the semi-continuity result for the measure-dimension mapping $dim^0_H$.\\

Proof of Theorem \ref{thm4}:

\begin{proof}
Accoriding to the definition of  $dim^0_H \nu$, for any small $\varepsilon>0$, we can find a measurable set $A$  with $\nu(A)>0$ and  $HD(A)\leq dim^0_H \nu+\varepsilon$. Since $\nu_n\stackrel{s}{\rightarrow}\nu$  as $n\rightarrow\infty$, we can guarantee that 
\begin{center}
$\nu_n(A)>0$
\end{center}
for any $n$ large enough. This implies that 
\begin{center}
$\limsup_{n\rightarrow\infty} dim^0_H \nu_n\leq dim^0_H \nu+\varepsilon$.
\end{center}
The proof is finished by letting $\epsilon\rightarrow 0$.

\end{proof}

We end the section by some results on comparing the dimensions between two comparable measures.
\begin{lemma}\label{lem3}
For any two probability measures $\nu_1, \nu_2\in\mathcal{M}(X)$, if $\nu_1$ is absolutely continuous with respect to $\nu_2$, then 
\begin{center}
$\dim_H^1 \nu_1\leq \dim_H^1  \nu_2$.
\end{center}
\end{lemma}
\begin{proof}
For $A\in\mathcal{A}$,  if $\nu_2(A)=1$, then $\nu_2(X\setminus A)=0$. Since $\nu_1$ is absolutely continuous with respect to $\nu_2$, we have $\nu_1(X\setminus A)=0$, so $\nu_1(A)=1$, so
\begin{center}
$dim_H^1 \nu_2=\inf\{HD(A): \nu_2(A)=1\}\geq\inf\{HD(A): \nu_1(A)=1\}=dim_H \nu_1$.
\end{center}
\end{proof}
One is recommended to compare the lemma with \cite[p220(a)]{MMR}. These conclusions again show duality of the two measure-dimension mappings.
\begin{corollary}\label{cor2}
For any two probability measures $\nu_1, \nu_2\in\mathcal{M}(X)$, if $\nu_1$ is equivalent to  $\nu_2$, then 
\begin{center}
$\dim_H ^1\nu_1= \dim_H^1 \nu_2$.
\end{center}

\end{corollary}

See also \cite[Lemma 3.1(3)]{HS}. Combining Young's Lemma and Corollary \ref{cor2}, we have the following result.

\begin{corollary}
For two equivalent measures $\nu_1, \nu_2\in\mathcal{M}(X)$ with one of them being ergodic with respect to some non-inverse-dimension-expanding transformation $T$ on $X$, we have
\begin{center}
$\dim_H^0 \nu_1= \dim_H^0  \nu_2=\dim_H^1 \nu_1= \dim_H^1  \nu_2$.
\end{center}
\end{corollary}

This corollary suggests a possible way to deal with the dimensions $\dim_H^0 \nu$ and $\dim_H^1  \nu$ of some (non-invariant) measure $\nu\in\mathcal{M}(X)$ simultaneously. That is, for some concerning measure $\nu\in\mathcal{M}(X)$, if one can find an ergodic measure $\nu_e\in\mathcal{M}(X)$ equivalent with $\nu$, then the dimension of  $\nu_e$ gives the dimension of $\nu$. As there are dynamical structure with respect to $\nu_e$, the dimension of  $\nu_e$ may be easier to be calculated. In case of non-existence of such an ergodic measure equivalent with $\nu$, one may resort to the ergodic decomposition of  $\nu$. The method still has some meaning for an invariant measure  $\nu$, though an invariant measure $\nu$ absolutely continuous with respect to an ergodic one $\nu_e$ necessarily satisfies
\begin{center}
$\nu=\nu_e$
\end{center}
according to \cite[P153, Remarks(1)]{Wal}.

\section{Applications to conformal iterated function systems}\label{sec6}

In this section we apply our semi-continuity results of the measure dimension mappings $dim_H^0$ and $dim_H^1$ to the conformal measures originated from conformal iterated function systems (CIFS), to release the Mauldin-Urba\'nski Theorem from the restriction of finite entropy. We also pay some attention to the correlation dimension $dim_C$ and modified correlation dimension $dim_{MC}$ of our concerning exploding measures in this section.

We first collect the following results from \cite{MU1}, which will be of our interest in this work.

\begin{theorem}[Mauldin-Urba\'nski]\label{thm20}

For a CIFS $S=\{s_i: X\rightarrow X\}_{i\in I}$ with a countable index set $I$, we have

\begin{itemize}
\item[(a)] $\Sigma_{i\in I} \| s'_i(x)\|^d \leq K^d$ and $\lim_{i\in I} diam(s_i(X))=0$ for a system with infinite index set $I$, in which $K\geq 1$ is the distortion parameter in BDP.

\item[(b)] In case of existence of a $t$-conformal measure $m$ on $J$, there exists a unique Borel probability measure $\mu$ on $I^\infty$ such that 
\begin{center}
$\mu([\omega])=\int_X |s'_\omega(x)|^t\ dm$
\end{center}
for any $\omega\in I^*$.

\item[(c)] In case of existence of a $t$-conformal measure $m$ on $J$, there exists a unique ergodic $\sigma$-invariant probability measure $\mu^*$ equivalent with $\mu$. The pushing-forward measure $m^*$ of $\mu^*$ under $\pi$ is an ergodic measure on $J$ equivalent with $m$.

\item[(d)] A $t$-conformal measure $m$ exists if and only if $P(t)=0$. In case of its existence, it is unique, and $m$-almost every point $x\in J$ has a unique pre-image $\pi^{-1}(x)$.

\item[(e)] When $P(t)=0$ has a solution $h$, then $h$ is the unique solution satisfying
\begin{center}
$h=HD(J)$. 
\end{center}
In the general case, we have
\begin{equation}\label{eq3}
HD(J)=\inf\{t\geq 0: P(t)<0\}=\sup_{F_I\subset I, \#F_I<\infty}\{HD(J_{F_I})\}\geq\inf\{t\geq 0: P(t)<\infty \}
\end{equation}
in which $J_{F_I}$ is the limit set of the finite  CIFS with finite indexing set $F_I\subset I$.

\item[(f)] If $S$ is a regular CIFS, then 
\begin{center}
$m_{F_I}\stackrel{w}{\rightarrow}m$, 
\end{center}
in which $m_{F_I}$ is the conformal measure with respect to the finite sub-system with finite index set $F_I\subset I$.
\end{itemize}

\end{theorem}

Due to Theorem \ref{thm20} (d), for $m$-almost every $x=\pi(\omega)=\pi(\omega_1\omega_2\cdots)\in J$,   define
\begin{center}
 $T(x)=s^{-1}_{\omega_1}(x)$. 
\end{center}
By pulling back  measures on $(J, \mathcal{B}_J)$ or pushing forward measures on $(I^\infty, \mathcal{B}_{I^\infty)}$ through $\pi$, we have the following commutative diagram,

\begin{center}
$
\begin{CD}
 E^\infty @>\sigma>> E^\infty\\
 @VV\pi V   @VV\pi V\\
J @>T>> J.
\end{CD}
$
\end{center}

Due to this reason, we will identify terminologies on the two dynamical systems $(J, \mathcal{B}_J, T, m^*)$ and $(I^\infty,  \mathcal{B}_{I^\infty}, \sigma, \mu^*)$. 

From now on we always assume  
\begin{center}
$S=\{s_i: X\rightarrow X\}_{i\in \mathbb{N}}$ 
\end{center}
is a regular countable infinite CIFS, without special declaration. Denote by  $h=HD(J)$. Due to Theorem \ref{thm20} (c), (d) and (e), there exists a unique $h$-conformal measure $m$ and a unique ergodic one $m^*$ (equivalent to $m$) supported on $J$.  We do not assume the system has finite entropy. Our strategy on proving Theorem \ref{thm3} is  to use its finite sub-systems to approximate the infinite CIFS $S$.  So let 
\begin{center}
$S_n=\{s_i: X\rightarrow X\}_{i\in \mathbb{N}_n}$
\end{center}
be its $n$-th finite sub-system with limit set $J_n$ for any $n\in\mathbb{N}$. Let $h_n=HD(J_n)$. According to Theorem \ref{thm20} (d) and (e), there exists a unique $h_n$-conformal measure $m_n$ supported on $J_n$ for any $n\in\mathbb{N}$. Let $m_n^*$ be the corresponding ergodic measure equivalent with $m_n$. Moreover, since $m_n^*$ (or $m_n$) has finite entropy, according to the Mauldin-Urba\'nski Theorem, we have  
\begin{center}
$h_n=dim_H^0 m_n=dim_H^1 m_n=dim_H^0 m_n^*=dim_H^1 m_n^*$ 
\end{center}
for any $n\in\mathbb{N}$. It is easy to see that 
\begin{center}
$J_n\subset J_{n+1}\subset J$ 
\end{center}
is a strictly increasing sequence of sets with respect to $n$, which implies that $\{h_n\}_{n=1}^\infty$ is a non-decreasing sequence. So by (\ref{eq3}), we have 
\begin{equation}\label{eq5}
\lim_{n\rightarrow\infty} h_n=h.
\end{equation}
According to Theorem \ref{thm20} (b), let $\mu_n$  be the unique Borel probability measure on  $\mathbb{N}_n^\infty$  such that 
\begin{center}
$\mu_n([\omega])=\int_X |s'_\omega(x)|^{h_n}\ dm_n$.
\end{center}
for any $\omega\in\mathbb{N}_n^\infty$. According to Theorem \ref{thm20} (c), let $\mu_n^*$  be the unique ergodic probability measure on $\mathbb{N}_n^\infty$ equivalent to $\mu_n$. All the measures can be extended trivially onto $J$, $X$, $\mathbb{N}_n^\infty$ or $\mathbb{N}^\infty$. 

%We achieve this by Corollary \ref{cor1}. In the following we gradually show that the sequence of measures  $\{ m_n\}_{n=1}^\infty$ satisfies the two conditions in Corollary \ref{cor1}.  
We first give a rough comparison between the two measures $m$ and $m_n$ on $J_n$ for $n\in\mathbb{N}$. Again we remind the readers that all our systems are supposed to be regular CIFS till the end of the section, without special declaration.
\begin{lemma}\label{lem4}
For any measurable $A\subset J_n$, we have
\begin{equation}\label{eq19}
m(A)\leq K^h m_n(A)
\end{equation}
for any $n\in\mathbb{N}$.
\end{lemma}

\begin{proof}
We first show the result is true on the projective cylinder sets on $J_n$, then transfer it to any closed measurable  subset $A\subset J_n$ through finer coverings. For $k\in\mathbb{N}$ and $\omega\in \mathbb{N}_n^k$, we have
\begin{center}
$\begin{array}{ll}
m(s_\omega(X)) & =\int_X|s_\omega'(x)|^hdm\leq \|s_\omega'\|^h\leq\|s_\omega'\|^{h_n}\leq\int_X|Ks_\omega'(x)|^{h_n}dm_n\\
 & \leq K^{h}\int_X|s_\omega'(x)|^{h_n}dm_n=K^{h}m_n(s_\omega(X)).
\end{array}$
\end{center}

Now for any closed measurable $A\subset J_n$, let 
\begin{center}
$B_k=\{\omega\in\mathbb{N}_n^k: s_\omega(X)\cap A\neq\emptyset\}$. 
\end{center}
Note that the decreasing sequence $\{\cup_{\omega\in B_k} s_\omega(X)\}_{k=1}^\infty$ tends to $A$ as $k\rightarrow\infty$, that is,
\begin{equation}\label{eq30}
\cap_{k=1}^\infty\cup_{\omega\in B_k}s_\omega(X)=A.
\end{equation}
To see this, if $x\in A$, let 
\begin{center}
$\omega=\pi^{-1}(x)\in \mathbb{N}^\infty$.
\end{center} 
Obviously $\omega|_k\in B_k$ for any $k\in\mathbb{N}$, so the right set in Equation (\ref{eq30}) is contained in the left set. Now if $x\in \cap_{k=1}^\infty\cup_{\omega\in B_k}s_\omega(X)$, then we can find $\omega_x\in \mathbb{N}^\infty$ such that $s_{\omega_x|_k}(X)$ contains at least one point $x_k\in A$. Obviously we have
\begin{center}
$\lim_{k\rightarrow\infty} x_k=x$.
\end{center}
Since $A$ is closed we have $x\in A$.   
Then we have
\begin{center}
$\begin{array}{ll}
m(A) & =\lim_{k\rightarrow\infty} m(\cup_{\omega\in B_k}s_\omega(X))=\lim_{k\rightarrow\infty} \Sigma_{\omega\in B_k}m(s_\omega(X))\\
 & \leq \lim_{k\rightarrow\infty} \Sigma_{\omega\in B_k}K^{h}m_n(s_\omega(X))=K^{h}\lim_{k\rightarrow\infty} m_n(\cup_{\omega\in B_k}s_\omega(X))=K^{h}m_n(A).
\end{array}$
\end{center}
Since both measures are regular, the inequality extends to any measurable set $A$.
 \end{proof}

The inequality (\ref{eq19}) suggests $m$ is absolutely continuous with respect to $m_n$ for any $n\in\mathbb{N}$ on $J_n$. It has ever been puzzled to the author whether they two are equivalent to each other on $J_n$, although it is visible if one knows \cite[P153, Remarks(1)]{Wal}. It turns out that they are completely singular with respect to each other, as implied in the proof of \cite[Theorem 4.3.10]{MU3}.

\begin{theorem}[Mauldin-Urba\'nski]\label{thm10}
For any subset $\tilde{I}\subset I$, consider the sub-CIFS 
\begin{center}
$S_{\tilde{I}}=\{s_i: X\rightarrow X\}_{i\in\tilde{I}}$.  
\end{center}
Let $m_{\tilde{I}}$ and $m_{\tilde{I}}^*$ be the corresponding conformal and ergodic measures on the limit set $J_{\tilde{I}}$ in case of their existence. Now if $\tilde{I}\subsetneq I$, then
\begin{center}
$m(J_{\tilde{I}})=m^*(J_{\tilde{I}})=0$,
\end{center} 
which implies $m_{\tilde{I}}$ (or $m_{\tilde{I}}^*$) is singular with respect to $m$ (or $m^*$).
\end{theorem}  
\begin{proof}
We can work in the symbolic space. Assume by contradiction that $m(J_{\tilde{I}})>0$ (or equivalently $m^*(J_{\tilde{I}})>0$). Then $\mu(\tilde{I}^\infty)>0$ considering the projective set. According to Theorem \ref{thm20} (c), we also have
\begin{center}
$\mu^*(\tilde{I}^\infty)>0$.
\end{center} 
Since $\mu^*$ is ergodic and $\sigma(\tilde{I}^\infty)=\tilde{I}^\infty$, so we get
\begin{equation}\label{eq31}
\mu^*(\tilde{I}^\infty)=1.
\end{equation}
Now choose $j\in I\setminus \tilde{I}$, since $\mu([j])=\int_X |s_j'(x)|^hdm>0$ and $\mu$ is equivalent to $\mu^*$, we have
\begin{equation}\label{eq32}
\mu^*([j])>0
\end{equation}
Then (\ref{eq31}) and (\ref{eq32}) contradict each other as $[j]\cap \tilde{I}=\emptyset$, which denies the assumption. 
\end{proof}

Considering the TV topology on the space of measures $\mathcal{M}(J)$, this implies the following result.
\begin{corollary}\label{cor9}
Any of the four sequences of measures 
\begin{center}
$\{m_n\}_{n=0}^\infty$, $\{m_n^*\}_{n=0}^\infty$, $\{\mu_n\}_{n=0}^\infty$
and $\{\mu_n^*\}_{n=0}^\infty$
\end{center}
does not converge under the TV topology on $\mathcal{M}(J)$.
\end{corollary}
\begin{proof}
This is because  for any two integers $n_1<n_2\in\mathbb{N}$,  we have 
\begin{center}
$\Vert m_{n_1}-m_{n_2}\Vert_{TV}=1$,
\end{center}
according to Theorem \ref{thm10}.
\end{proof}

\begin{rem}
The conclusion can be checked through direct computations in some cases. For example,  let $S=\{s_i: X\rightarrow X\}_{i\in \mathbb{N}}$ be a regular CIFS with  $s_i$ being a \emph{similitude} (\cite[Definition 3.9]{Kae})) for any $i\in\mathbb{N}$, that is,
\begin{center}
$s_i'(x)=a_i$
\end{center}
for some $ 0<a_i<1$ and any $i\in\mathbb{N}$. Let $h_{n_2}=HD(J_{n_2})$. One can get that
\begin{center}
$
\begin{array}{lll}
m_{n_2}(J_{n_2}\setminus J_{n_1}) & =& \sum_{i=n_1+1}^{n_2} a_i^{h_{n_2}}+(1-\sum_{i=n_1+1}^{n_2} a_i^{h_{n_2}})\sum_{i=n_1+1}^{n_2} a_i^{h_{n_2}}\\
 & & +(1-\sum_{i=n_1+1}^{n_2} a_i^{h_{n_2}})^2\sum_{i=n_1+1}^{n_2} a_i^{h_{n_2}}+\cdots \\
 & = & 1,
\end{array}
$
\end{center}
and hence $m_{n_2}(J_{n_1})=0$ for any $n_1<n_2$.
\end{rem}

Corollary \ref{cor9} is a depressing result to us. However, luckily, when we switch to the setwise topology on $\mathcal{M}(J)$, there is some nice change happening.  The following theorem  strengthens the conclusion Theorem \ref{thm20} (f) in \cite{MU1}. 

\begin{theorem}\label{lem2}
Under the setwise topology on $\mathcal{M}(J)$, we have
\begin{center}
$m_n\stackrel{s}{\rightarrow}m$ as $n\rightarrow\infty$.
\end{center}
\end{theorem}
\begin{proof}
We first show that 
\begin{equation}\label{eq24}
\lim_{n\rightarrow\infty} m_n(A)=m(A)
\end{equation}
for  $A$ being a cylinder set. We do  this by induction on the level of the cylinder sets. First we establish the initiative step. Since $X$ is compact and $s_i(x)$ is $C^{1+\theta}$ for any $i\in\mathbb{N}$, due to BDP, (\ref{eq5}) and Theorem \ref{thm20} (f), we claim that
\begin{equation}
\begin{array}{ll}\label{eq33}
\lim_{n\rightarrow\infty}\mu_n([i]) & =\lim_{n\rightarrow\infty} \int_X |s'_i(x)|^{h_n}dm_n=\lim_{n\rightarrow\infty} \int_X |s'_i(x)|^{h}dm_n\\
 & =\int_X |s'_i(x)|^{h}dm=\mu([i])
\end{array}
\end{equation}
for any $i\in\mathbb{N}$. 

To see this, first, due to BDP and (\ref{eq5}), the sequence of functions 
\begin{center}
$|s'_{i}(x)|^{h_n}\rightarrow |s'_{i}(x)|^h$
\end{center}
uniformly as $n\rightarrow\infty$ on whole $X$, that is, for any small $\epsilon>0$, there exists $N_1\in\mathbb{N}$ large enough, such that 
\begin{center}
$0<|s'_{i}(x)|^{h_n}-|s'_{i}(x)|^h<\cfrac{\epsilon}{2}$
\end{center}
for any $n\geq N_1$. By integration with respect to $m_n$, we have
\begin{center}
$0<\int_X |s'_{i}(x)|^{h_n}dm_n-\int_X |s'_{i}(x)|^hdm_n<\cfrac{\epsilon}{2}$
\end{center}
for any $n\geq N_1$. According to Theorem \ref{thm20} (f), there exists $N_2\in\mathbb{N}$ large enough, such that 
\begin{center}
$|\int_X |s'_{i}(x)|^h dm_n-\int_X |s'_{i}(x)|^hdm|<\cfrac{\epsilon}{2}$
\end{center}
for any $n\geq N_2$. Now let $N=\max\{N_1, N_2\}$. Then for any $n\geq N$, we have
\begin{center}
$
\begin{array}{ll}
& |\mu_n([i])-\mu([i])| \\
= & |\int_X |s'_{i}(x)|^{h_n} dm_n-\int_X |s'_{i}(x)|^hdm| \\
  \leq  &  |\int_X |s'_{i}(x)|^{h_n}dm_n-\int_X |s'_{i}(x)|^hdm_n|+|\int_X |s'_{i}(x)|^h dm_n-\int_X |s'_{i}(x)|^hdm| \\
 <   & \epsilon,
\end{array}
$
\end{center}
which justifies (\ref{eq33}). 

Now suppose (\ref{eq24}) holds on any level $k-1$ cylinder set for $k\geq 2$, we will show that it is still true on any level $k$ cylinder set. For a $k$-word $\omega=\omega_1\omega_2\cdots\omega_k\in\mathbb{N}^k$, set $B_\omega=\pi([\omega])$. Then
\begin{center}
$m_n(B_\omega)=\mu_n([\omega])=\int_{\pi([\omega_2\omega_3\cdots \omega_k])}|s'_{\omega_1}(x)|^{h_n}dm_n $
\end{center}
for $n$ large enough. In the following we always assume $n$ is large enough. Note that
\begin{center}
$m(B_\omega)=\mu([\omega])=\int_{\pi([\omega_2\omega_3\cdots \omega_k])}|s'_{\omega_1}(x)|^{h}dm$.
\end{center}
Again due to BDP, (\ref{eq5}) and Theorem \ref{thm20} (f), we have
\begin{equation}\label{eq26}
 \lim_{n\rightarrow\infty} \int_X |s'_{\omega_1}(x)|^{h_n}dm_n= \lim_{n\rightarrow\infty} \int_X |s'_{\omega_1}(x)|^{h}dm_n=\int_X |s'_{\omega_1}(x)|^{h}dm.
\end{equation}
Now by the inductive assumption on $k-1$ level cylinder sets, we have
\begin{equation}\label{eq27}
\begin{array}{ll}
\lim_{n\rightarrow\infty}\mu_n([\omega_2\omega_3\cdots \omega_k]) & =\lim_{n\rightarrow\infty} \int_X 1_{\pi([\omega_2\omega_3\cdots \omega_k])}dm_n \\
 & = \mu([\omega_2\omega_3\cdots\omega_k]) \\
 & = \int_X 1_{\pi([\omega_2\omega_3\cdots\omega_k])}dm.
\end{array}
\end{equation}
Combining (\ref{eq26}) and (\ref{eq27}) together, by an $\epsilon-N$ argument, we have
\begin{center}
$
\begin{array}{ll}
& \lim_{n\rightarrow\infty} \int_X |s'_{\omega_1}(x)|^{h_n}\cdot 1_{\pi([\omega_2\omega_3\cdots \omega_k])} dm_n \\
= & \lim_{n\rightarrow\infty} \int_X |s'_{\omega_1}(x)|^{h}\cdot 1_{\pi([\omega_2\omega_3\cdots \omega_k])}dm_n \\
= & \int_X |s'_{\omega_1}(x)|^{h}\cdot 1_{\pi([\omega_2\omega_3\cdots \omega_k])}dm,
\end{array}
$
\end{center}
that is,
\begin{center}
$\lim_{n\rightarrow\infty}\mu_n([\omega_1\omega_2\cdots \omega_k]) =\mu([\omega_1\omega_2\cdots \omega_k])$,
\end{center} 
which completes the inductive step. 

At last, by replaying the argument at the end of proof of Lemma \ref{lem4}, we can first extend (\ref{eq24}) from cylinder sets to closed measurable ones, and then to all measurable  ones via regularity of the measures, or we can directly reach the conclusion in virtue of \cite[Theorem 2.3]{FKZ}. 
\end{proof}

Now we are in a position to prove Theorem \ref{thm3}.\\

Proof of Theorem \ref{thm3}:
\begin{proof}
By Theorem \ref{thm1}, Theorem \ref{thm4} and Theorem \ref{lem2}, we have

\begin{equation}\label{eq34}
\limsup_{n\rightarrow\infty} dim^0_H m_n\leq dim^0_H m\leq dim^1_H m\leq \liminf_{n\rightarrow\infty} dim_H^1 m_n.
\end{equation}
Considering the corresponding equivalent ergodic measures $m_n^*$ and $m^*$, due to  Corollary \ref{cor2} and \cite[Lemma 3.1(3)]{HS}, we have
\begin{equation}\label{eq35}
dim_H^0 m_n= dim_H^0 m_n^*=dim_H^1 m_n= dim_H^1 m_n^*
\end{equation}
and 
\begin{equation}\label{eq36}
dim_H^0 m= dim_H^0 m^*=dim_H^1 m=dim_H^1 m^*.
\end{equation}

Finally, Combining all the  formulas (\ref{eq34}) (\ref{eq35}) (\ref{eq36}) and (\ref{eq5}),  we justify all the equalities in (\ref{eq42}), except the last one. For the last equality in (\ref{eq42}), see Remark \ref{rem2}.

\end{proof}

For two infinite words $\omega,\omega'\in\mathbb{N}^\infty$, let $k$ be the least integer such that $\omega_k\neq \omega'_k$ in case of $\omega\neq\omega'$. Define the \emph{comparison metric} $\rho_c$ on $\mathbb{N}^\infty$ by

\begin{equation}\label{eq38}
\rho_c(\omega,\omega')=\left\{
\begin{array}{ll}
e^{1-k} & \mbox{ if } \omega\neq\omega',\\
0 & \mbox{ if } \omega=\omega'.\\
\end{array}
\right.
\end{equation}
Then  $(\mathbb{N}^\infty, \rho_c)$ is a compact metric space of diameter $1$.  By similar method as to dealing with the sequence of measures $\{m_n\}_{n=1}^\infty$, considering Corollary \ref{cor2} and \cite[Lemma 3.1(3)]{HS}, we can show the following result.
\begin{proposition}
Under the above notations, we have
\begin{center}
$\mu_n\stackrel{s}{\rightarrow}\mu$ as $n\rightarrow\infty$ on $\mathcal{M}(\mathbb{N}^\infty)$
\end{center}
and
\begin{center}
$
\begin{array}{ll}
& dim_H^0 \mu_n=dim_H^0 \mu_n^*=dim_H^1 \mu_n=dim_H^1 \mu^*_n \\
\rightarrow & dim_H^0 \mu=dim_H^0 \mu^*=dim_H^1 \mu=dim_H^1 \mu^*
\end{array}
$
\end{center}
as $n\rightarrow\infty$.
\end{proposition}

It is an interesting question to ask the correlation dimension and the modified correlation dimension  of the above measures. Considering Theorem \ref{thm5}, the strategy for dealing with $dim_H^0$ and $dim_H^1$ loses its power now. It turns out that for regular CIFS, the $h$-conformal measure $m$ and the corresponding ergodic measure $m^*$ have the same correlation dimension and the modified correlation dimension with their upper and lower Hausdorff dimension. 

\begin{theorem}\label{thm14}
For any regular CIFS $S=\{s_i: X\rightarrow X\}_{i\in I}$ with a countable index set $I$, we have
\begin{center}
$dim_C m=dim_C m^*=dim_{MC} m=dim_{MC} m^*=HD(J)=h$
\end{center}
with $J$ being the limit set of the system.
\end{theorem}  
\begin{proof}
In case of $\# I$ being finite, the conclusion follows from \cite[Lemma 3.14]{MU1} directly. In case of $\#I$ being infinite, we will continue to utilize the proof of \cite[Lemma 3.14]{MU1} by  playing a small trick on the limit set $J$. Without loss of generality suppose $I=\mathbb{N}$ now. First, according to Theorem \ref{thm20} (a), for a small fixed positive $r>0$, choose $l\in\mathbb{N}$ large enough such that 
\begin{center}
$\sum_{i> l}\Vert s_i\Vert^h=o(r^h)$ and $\gamma^l diam(X)<r$,
\end{center}
in which $diam(X)$ is the diameter of the space $X$. Let $k$ be the smallest integer such that
\begin{equation}\label{eq37}
\Vert s_\omega'\Vert<\cfrac{r}{2}
\end{equation}
for any $\omega\in\mathbb{N}_l^k$ of length $k$. Let $\tilde{J}=\cup_{\omega\in\mathbb{N}_l^k} s_\omega(X)$. Now split the limit set $J$ into the following two sets:
\begin{center}
$J=\tilde{J}\cup(J\setminus\tilde{J})$
\end{center}
The set $\tilde{J}$ dominates $J\setminus\tilde{J}=\cup_{i=1}^k \cup_{\omega\in\mathbb{N}^k, \omega_i>l} s_\omega(X)$ in the sense of the $h$-conformal measure $m$ as
\begin{center}
$m(J\setminus\tilde{J})\leq \sum_{j=0}^{k-1} \gamma^j\sum_{i>l}\Vert s_i\Vert^h\leq \cfrac{1}{1-\gamma}\sum_{i>l}\Vert s_i\Vert^h=o(r^h)$.
\end{center}

In the following we focus on the dominating set $\tilde{J}$. Now for any $x=\pi(\omega)\in\tilde{J}$, let 
\begin{center}
$\xi=\inf\{\vert s_i'\vert: 0\leq i\leq l\}>0$.
\end{center} 
Note that due to  (\ref{eq37}), the smallest integer $n$ such that $s_{\omega|_n}\subset B(x,r)$ must satisfy
\begin{center}
$n\leq k$.
\end{center}
Again due to  (\ref{eq37}), any minimal word $\omega\in \mathbb{N}_l^*$  needs to satisfy
\begin{center}
$|\omega|\leq k$.
\end{center}
So the proof of \cite[Lemma 3.14]{MU1} is applicable in this case, which guarantees that there exists $C>1$ such that
\begin{center}
$\cfrac{1}{C}\leq \cfrac{m(B(x,r))}{r^h}\leq C$
\end{center} 
for any $x\in \tilde{J}$ and $r$ small enough. So
\begin{center}
$C^{-1}r^h\leq \int_{X} m(B(x,r))dm=\int_{\tilde{J}} m(B(x,r))dm+\int_{J\setminus\tilde{J}} m(B(x,r))dm\leq Cr^h+o(r)$.
\end{center}
Taking logarithm, dividing by $\log r$ and then let $r\rightarrow 0$ in the above formula, we get 
\begin{center}
$dim_C m=dim_{MC} m=HD(J)=h$.
\end{center}
The equality on $dim_C m^*$ and $dim_{MC} m^*$ follows from \cite[Theorem 2.6, Theorem 2.10]{MMR} as well as equivalence of the two measures $m$ and $m^*$.
\end{proof}

\section{Applications to conformal graph directed Markov systems}\label{sec8}

In this section we apply our semi-continuity results of the measure dimension mappings $dim_H^0$ and $dim_H^1$ to the conformal graph directed Markov systems (CGDMS), to overcome the difficulty caused by the appearance of exploding entropy of concerning measures in this infinite IFS theory. 

We first recall some key techniques on judging the weak convergence of sequences of measures on metric spaces. In case that there are no obvious target limit measure for a sequence of probability measures $\{\nu_n\in\mathcal{M}(X)\}_{n=1}^\infty$ with $X$ being metrizable, the Prohorov’s Theorem is a very important tool to judge its weak convergence. Now we recall some definitions in \cite[Chapter 1, Section 5]{Bil1}.

\begin{definition}
A family of probability measures $\prod\subset \mathcal{M}(X)$ is said to be \emph{relatively compact} if  every sequence of measures from $\prod$ contains a weakly convergent subsequence. It is called \emph{tight} if for every $\epsilon>0$, there is a compact set $X_c\subset X$, such that 
\begin{center}
$\nu(X_c)>1-\epsilon$ 
\end{center}
for any $\nu\in\prod$. 
\end{definition}

The two notions are closely related. As one can expect, tightness of a family of measures is usually a property involving only co-finitely many members of the family, instead of all of them in case the topology of $X$ is well-behaved.

\begin{proposition}
Suppose $X$ is a separable and complete metric space. Then for a family of probability measures $\prod\subset \mathcal{M}(X)$, if for every $\epsilon>0$, there exists a compact set $X_\epsilon\subset X$, such that 
\begin{center}
$\nu(X_\epsilon)>1-\epsilon$ 
\end{center}
for all but finitely many $\nu\in\prod$, then the family $\prod$ is tight. 
\end{proposition}
\begin{proof}
Suppose that for a fixed $\epsilon>0$, we have
\begin{center}
$\nu(X_\epsilon)>1-\epsilon$ 
\end{center}
for any $\nu\in\prod\setminus\{\nu_1, \nu_2, \cdots, \nu_n\}$ with $\{\nu_1, \nu_2, \cdots, \nu_n\}\subset\prod$, $n\in\mathbb{N}$. Since $X$ is separable and complete, apply \cite[Theorem 1.3]{Bil1}, there exists a sequence of compact subsets $X_1, X_2, \cdots, X_n$ in $X$, such that 
\begin{center}
$\nu_i(X_i)>1-\epsilon$
\end{center}
for any $1\leq i\leq n$. Now let $X_c=X_\epsilon\cup_{1\leq i\leq n} X_i$. It is a compact subset of $X$ satisfying
\begin{center}
$\nu(X_c)>1-\epsilon$
\end{center}
for any $\nu\in\prod$.
\end{proof}

The Prohorov’s Theorem contains two conclusions converse to each other, see \cite[Theorem 5.1, Theorem 5.2]{Bil1}
\begin{Prohorov's Theorem}
Let $X$ be a metric space, then
\begin{enumerate}[(a).]
\item If a family of probability measures $\prod\subset \mathcal{M}(X)$ is tight, then it is relatively compact.
\item In case of $X$ being separable and complete, if a family of probability measures $\prod\subset \mathcal{M}(X)$ is relatively compact, then it is tight.
\end{enumerate}
\end{Prohorov's Theorem}

An important corollary of the conclusion (a) above is \cite[P59, Corollary]{Bil1}.
\begin{corollary}[Billingsley]\label{cor3}
Let $X$ be a metric space. If a sequence of measures $\{\nu_n\in\mathcal{M}(X)\}_{n=1}^\infty$ is tight, and every weakly convergent subsequence of it converges to the same metric $\nu\in\mathcal{M}(X)$, then \begin{center}
$\nu_n\stackrel{w}{\rightarrow}\nu$ 
\end{center}
as $n\rightarrow\infty$.
\end{corollary}
The result is critical to show the weak convergence of dimensions of our concerning sequences of measures in our Theorem \ref{thm11}.

Now we turn to GDMS. For the dynamical system $(E^{\infty,A},\sigma)$ and a (potential) function $f: E^{\infty,A}\rightarrow \mathbb{R}$, define the \emph{topological pressure} of $f$ with respect to $(E^{\infty,A},\sigma)$ as
\begin{equation}\label{eq55}
P(f)=\lim_{n\rightarrow\infty}\cfrac{1}{n}\log Z_n(f)=\inf_{n\in\mathbb{N}}\{\cfrac{1}{n}\log Z_n(f)\},
\end{equation}
in which 
\begin{center}
$Z_n(f)=\sum_{\omega\in E^{n,A}}\exp\big(\sup_{\omega'\in[\omega]}\sum_{i=0}^{n-1}f(\sigma(\omega'))\big)$
\end{center}
is the $n$-th \emph{partition function} for the potential $f$ with respect to $(E^{\infty,A},\sigma)$ for any $n\in\mathbb{N}$. The above equality (\ref{eq55}) holds since $\{Z_n(f)\}_{n\in\mathbb{N}}$ is a \emph{subadditive} sequence, see also \cite[Proposition 2.4]{Kae} and \cite[(3.10)]{Fal2}. Recall the comparison metric $\rho_c$ in (\ref{eq38}).  A potential $f: E^{\infty,A}\rightarrow \mathbb{R}$ is called \emph{H\"older continuous with exponent} $\alpha$ if
\begin{center}
$V_\alpha(f)=\sup_{\omega,\omega'\in E^{\infty,A}, \omega\neq\omega'}\{\cfrac{f(\omega)-f(\omega')}{(\rho_c(\omega,\omega'))^\alpha}\}< \infty$.
\end{center} 
It is called \emph{summable} if
\begin{center}
$\sum_{e\in E}\exp(\sup\{f(\omega): \omega\in [e]\})<\infty$.
\end{center}
Let $C^b(E^{\infty,A})$ be the Banach space of all bounded continuous functions on $E^{\infty,A}$ equipped with the supreme norm $||\cdot||_{\infty}$. Let $f$ be a bounded  summable H\"older continuous function with exponent $\alpha$ from now on. 

\begin{definition}
The \emph{Ruelle-Perron-Frobenius} (or \emph{transfer}) \emph{operator} 
\begin{center}
$\mathcal{L}_f: C^b(E^{\infty,A})\rightarrow C^b(E^{\infty,A})$
\end{center}
associated with $f$ in this context is defined to be:
\begin{center}
$\mathcal{L}_f(g)(\omega)=\sum_{e\in E, A_{e\omega_1}=1}\exp(f(e\omega))g(e\omega)$
\end{center} 
for any $g\in C^b(E^{\infty,A})$ and $\omega\in E^{\infty,A}$. 
\end{definition}

Refer to \cite{Rue}. Its conjugate operator $\mathcal{L}_f^*$
acts on the dual space of $C^b(E^{\infty,A})$ as 
\begin{center}
$\mathcal{L}_f^*(\nu)(g)=\nu(\mathcal{L}_f(g))=\int \mathcal{L}_f(g)d\nu$
\end{center} 
for any $\nu\in\tilde{\mathcal{M}}(E^{\infty,A})$ of finite measures on  $E^{\infty,A}$ and $g\in C^b(E^{\infty,A})$.

\begin{rem}
Let $C^{bm}(E^{\infty,A})$ be the Banach space of all bounded measurable functions on $E^{\infty,A}$ equipped with the $L^p$ norm $||\cdot||_{p}$ for some $p\in\mathbb{N}$. Since $f$ is bounded measurable and summable, for  any bounded measurable function $g\in C^{bm}(E^{\infty,A})$, the series of functions 
\begin{center}
$\sum_{e\in E, A_{e\omega_1}=1}\exp(f(e\omega))g(e\omega)$
\end{center}
converges pointwisely for  any $\omega\in E^{\infty,A}$ and the limit function is bounded.  Then according to \cite[Corollary 8.9]{Sch}, the limit function is also measurable. In this way we extend the Ruelle-Perron-Frobenius operator from  $C^b(E^{\infty,A})$ to a larger space $C^{bm}(E^{\infty,A})$, which enables us to deal with those discontinuous functions in $C^{bm}(E^{\infty,A})$. But in this work the space $C^b(E^{\infty,A})$ is enough for us to deal with our concerns.
\end{rem}
Now we show that the Ruelle-Perron-Frobenius operator $\mathcal{L}_f^*$ acts continuously on the space $\tilde{\mathcal{M}}(E^{\infty,A})$ with respect to the weak topology. For a GDMS $S$ with infinite edges $E=\{e_1, e_2, \cdots, \}$, let
\begin{center}
$E_n=\{e_1, e_2, \cdots, e_n\}$
\end{center}
be its $n$-th truncation  for any $n\in\mathbb{N}$.  For a bounded  summable H\"older continuous function $f$, let $f_n=f|_{E_n^{\infty,A|_{E_n\times E_n}}}$.

\begin{lemma}\label{lem11}
For a bounded  summable H\"older continuous function $f$ with exponent $\alpha$, let $\{\nu_n\}_{n\in\mathbb{N}}\subset\tilde{\mathcal{M}}(E^{\infty,A})$ be a sequence of finite measures such that
\begin{center}
$\nu_n\stackrel{w}{\rightarrow}\nu$ 
\end{center} 
as $n\rightarrow\infty$ for some $\nu\in\tilde{\mathcal{M}}(E^{\infty,A})$. Then we have
\begin{center}
$\mathcal{L}_{f_n}^*(\nu_n)\stackrel{w}{\rightarrow} \mathcal{L}_f^*(\nu)$
\end{center}
as $n\rightarrow\infty$.
\end{lemma}

\begin{proof}
For any bounded continuous function $g\in C^b(E^{\infty,A})$, we have
\begin{center}
$\mathcal{L}_f^*(\nu_n)(g)=\nu_n(\mathcal{L}_f(g))=\int \mathcal{L}_f(g)d\nu_n=\int\sum_{e\in E, A_{e\omega_1}=1}\exp(f(e\omega))g(e\omega)d\nu_n$
\end{center}
and 
\begin{center}
$\mathcal{L}_{f_n}^*(\nu_n)(g)=\nu_n(\mathcal{L}_{f_n}(g))=\int \mathcal{L}_{f_n}(g)d\nu_n=\int\sum_{e\in E_n, A_{e\omega_1}=1}\exp(f_n(e\omega))g(e\omega)d\nu_n$.
\end{center}
Since $f$ is bounded,  summable and H\"older continuous, for any $\epsilon>0$ and $n$ large enough, we have
\begin{equation}\label{eq39}
|\mathcal{L}_{f_n}^*(\nu_n)(g)-\mathcal{L}_f^*(\nu_n)(g)|<\epsilon.
\end{equation}

As $\mathcal{L}_f$ preserves $C^b(E^{\infty,A})$ and $\nu_n\stackrel{w}{\rightarrow}\nu$, we have
\begin{equation}\label{eq40}
\int \mathcal{L}_f(g)d\nu_n\rightarrow\int \mathcal{L}_f(g)d\nu=\nu(\mathcal{L}_f(g))=\mathcal{L}_f^*(\nu)(g)
\end{equation}
as $n\rightarrow\infty$. So the conclusion follows from a combination of (\ref{eq39}) and (\ref{eq40}).
\end{proof}

\begin{rem}
Note that even if we assume $\nu, \nu_n\in\mathcal{M}(E^{\infty,A})$ for any $n\in\mathbb{N}$, the measures 
\begin{center}
$\{\mathcal{L}_f^*(\nu_n)\}_{n\in\mathbb{N}}\cup\{\mathcal{L}_f^*(\nu)\}\subset \tilde{\mathcal{M}}(E^{\infty,A})$
\end{center}
are possibly to be not probability ones. This is sure for any eigenmeasure with eigenvalue not equal to $1$.
\end{rem}

Our ultimate concern will be on the $t$-conformal measures supported on the limit set  $J$ of a CGDMS. However, following Mauldin-Urba\'nski, we deal with it in a more general environment-the $F$-conformal measures induced by a summable H\"older family $F$ of functions. 

\begin{definition}
A collection of functions 
\begin{center}
$F=\{f^{(e)}: X_{t(e)}\rightarrow\mathbb{R}\}_{e\in E}$
\end{center}
is called a \emph{summable H\"older family of functions of order $\alpha$} if
\begin{center}
$\sum_{e\in E} \Vert exp(f^{(e)})\Vert_{\infty}<\infty $
\end{center}
and
\begin{center}
$\sup_{n\in\mathbb{N}}V_n(F)=\sup_{n\in\mathbb{N}}\Big\{\sup_{\omega\in E^n}\sup_{x,y\in X_{t(\omega)}}\{\cfrac{|f^{(\omega_1)}(s_{\sigma(\omega)}(x))-f^{(\omega_1)}(s_{\sigma(\omega)}(y))|}{(e^{-(n-1)})^\alpha}\}\Big\}<\infty$.
\end{center}
\end{definition}

Define its \emph{topological pressure} as
\begin{center}
$P(F)=\lim_{n\in\mathbb{N}}\cfrac{1}{n}\log\Big(\sum_{\omega\in E^n}\Vert exp(\sum_{j=1}^n f^{(\omega_j)}\circ s_{\sigma^j(\omega)})\Vert_{\infty}\Big)$.
\end{center}

We call the induced potential from $E^\infty$ to $\mathbb{R}$
\begin{center}
$f(\omega)=f^{(\omega_1)}(\pi(\sigma(\omega)))$
\end{center}
the \emph{amalgamated function} of the family $F$. 

Mauldin and Urba\'nski showed that 
\begin{center}
$P(F)=P(f)$.
\end{center}
They also showed that if $F$ is a  H\"older family of order $\alpha$, then its induced amalgamated function is also a H\"older function of order $\alpha$. Now we show the converse is also true.
\begin{proposition}
If  the amalgamated function $f:E^\infty\rightarrow\mathbb{R}$ of a family $F=\{f^{(e)}: X_{t(e)}\rightarrow\mathbb{R}\}_{e\in E}$ is a H\"older function of order $\alpha$, then family $F$ is a H\"older family of order $\alpha$.
\end{proposition} 
\begin{proof}
Not that for any $\omega\in E^n$ and $x,y\in X_{t(\omega)}$, we have
\begin{center}
$\rho_c(\omega\pi^{-1}(x),\omega\pi^{-1}(y))\leq e^{-(n-1)\alpha}$.
\end{center}
Then
\begin{center}
$
\begin{array}{ll}
& \cfrac{|f^{(\omega_1)}(s_{\sigma(\omega)}(x))-f^{(\omega_1)}(s_{\sigma(\omega)}(y))|}{(e^{-(n-1)})^\alpha}\\
= & \cfrac{|f(\omega\pi^{-1}(x))-f(\omega\pi^{-1}(y))|}{\rho_c(\omega\pi^{-1}(x),\omega\pi^{-1}(y))}\cdot\cfrac{\rho_c(\omega\pi^{-1}(x),\omega\pi^{-1}(y))}{e^{-\alpha(n-1)}}\\
\leq & V_\alpha(f)
\end{array}
$
\end{center}
\end{proof}
 
\begin{rem}
Note that  if the GDMS has the property that almost every $x\in X$ has a unique representation $\pi^{-1}(x)\in E^\infty$ with respect to some measure $\nu$ on $X$, then the correspondence between summable H\"older families of functions and summable H\"older potentials on $E^\infty$  is  $1-1$ by neglecting some domain of null measure. For a summable H\"older potential $f: E^\infty\rightarrow\mathbb{R}$, simply define its induced H\"older family of functions to be
\begin{equation}\label{eq41}
F=\{f^{(e)}(x)=f(e\pi^{-1}(x)),\ x\in X_{t(e)}\}_{e\in E}.
\end{equation}
If we require the measure $\nu$ to be supported on $J$ and the GDMS is a CGDMS, then the H\"older potentials can be assumed to be only on $E^{\infty,A}$, with (\ref{eq41}) defined essentially for words $\pi^{-1}(x)$ such that $e\pi^{-1}(x)\in E^{\infty,A}$. This is our main interested case in this work. From this point of view, the description of the dynamics of CGDMS by summable H\"older families of functions and summable H\"older potentials is parallel to each other.   
\end{rem}

Now we define the concerning measures we study in this section. 

\begin{definition}\label{def3}
For a summable H\"older family of functions $F$, a Borel probability measure $m_F\in\mathcal{M}(X)$ is  called $F$-\emph{conformal} if it satisfies the following conditions:\\

 (a1). $m_F(J)=1$.\\
 
(a2). For any $e\in E$ and any measurable set $B\subset X_{t(e)}$, the following holds,
\begin{center}
$m_F(s_\omega(B))=\int_B \exp(f^{(e)}-P(F))dm_F$.
\end{center}

 (a3). For any distinct pair of edges $e,e'\in E$, 
\begin{center}
 $m_F\big(s_e(X_{t(e)})\cap s_{e'}(X_{t(e')})\big)=0$.
\end{center} 

\end{definition}

For the existence and uniqueness of the $F$-conformal measure $m_F$ for a \emph{conformal}-\emph{like} GDMS with finitely primitive incidence matrix $A$ and a summable H\"older family $F$ , see \cite[Theorem 3.2.3]{MU3}. In this case we denote its pull-back measure on $E^{\infty, A}$  by 
\begin{center}
$\mu_F=m_F\circ\pi$. 
\end{center}
It is the unique eigenmeasure of the conjugate Ruelle-Perron-Frobenius operator $\mathcal{L}_f^*$ with eigenvalue $e^{P(F)}$ ($f$ is the amalgamated function of $F$). The ergodic measure equivalent with $\mu_F$ is denoted by $\mu_F^*$. So 
\begin{center}
$m_F^*=\mu^*_F\circ\pi^{-1}$ 
\end{center}
is the ergodic measure equivalent with $m_F$ on $J$, in virtue of \cite[Corollary 2.7.5]{MU3}. Note that our notations are different from ones therein as we are trying to maintain consistance of notations with measures in Section \ref{sec6}. From now on we focus on the asymptotic behaviour of the finite sub-systems of an infinite GDMS $S$. Let 
\begin{center}
$V_n=\{v\in V: \mbox{\ there exists some\ } e\in E_n \mbox{\ such that\ } i(e)=v \mbox{\ or\ } t(e)=v\}$.
\end{center}
For a summable H\"older family $F$, let 
\begin{center}
$F_n=\{f^{(e)}: X_{t(e)}\rightarrow\mathbb{R}\}_{e\in E_n}$
\end{center}
be its $n$-th sub-family. Obviously they are also summable H\"older families.  In case of a GDMS being finitely irreducible or finitely primitive, let 
\begin{center}
$\tilde{\Lambda}=\{e\in E: \mbox{\ there exists a finite word\ } \omega\in\Lambda, \mbox{\ such that\ } \omega_i=e \mbox{\ for some\ } 1\leq i\leq |\omega|\}$.
\end{center}  We first show the following result.
\begin{lemma}\label{lem12}
For a conformal-like GDMS 
\begin{center}
$S=\big\{V, E, A, t, i, \{X_v\}_{v\in V}, \{s_e\}_{e\in E}, \gamma\big\}$ 
\end{center}
with infinite edges $E=\{e_1, e_2, \cdots, \}$, a finitely primitive incidence matrix $A$ and a summable H\"older family $F$, let $\Lambda\subset E^{*,A}$ be a finite set which witnesses its finite primitiveness. Then there exists an integer $N_\Lambda\in\mathbb{N}$, such that for any $n\geq N_\Lambda$, there exists a unique $F_n$-conformal measure  $m_{F_n}$ for the sub-GDMS 
\begin{center}
$S_n=\big\{V_n, E_n, A|_{E_n\times E_n}, t, i, \{X_v\}_{v\in V_n}, \{s_e\}_{e\in E_n}, \gamma\big\}$. 
\end{center} 

\end{lemma}
\begin{proof}
The sub-GDMS $S_n$ inherits the conformal-like  property of $S$ and $F_n$ inherits the summable and H\"older property of the family  $F$ obviously. Now let $N_\Lambda$ be the least integer such that 
\begin{center}
$E_n\supset \tilde{\Lambda}$.
\end{center}   
Then for $n\geq N_\Lambda$, $\Lambda\subset E_n^{*,A|_{E_n\times E_n}}$ witnesses the  finite primitiveness of $S_n$. So for any $n\geq N_\Lambda$, there is a unique $F_n$-conformal measure  $m_{F_n}$ for $S_n$ according to \cite[Theorem 3.2.3]{MU3}.
\end{proof}
Note that according to \cite[Theorem 2.1.5, Proposition 3.1.4]{MU3}, the sequence of topological pressure $\{P(F_n)\}_{n\in\mathbb{N}}$ of the sub-GDMS $S_n$ converges to $P(F)$. 

\begin{lemma}\label{lem13}
The sequence of conformal measures $\{m_{F_n}\}_{n\geq N_\Lambda}$ is tight on $J$.
\end{lemma}
\begin{proof}
First, since the sequence of topological pressure $\{P(F_n)\}_{n\in\mathbb{N}}$ increases to $P(F)$ according to \cite[Theorem 2.1.5]{MU3}, then the sequence $\{\exp(-P(F_n))\}_{n\in\mathbb{N}}$ is bound, say 
\begin{center}
$\sup_{n\in\mathbb{N}}\{\exp(-P(F_n))\}\leq M$. 
\end{center}
As $F$ is summable, for any $\epsilon>0$, there is some $N_\epsilon\in\mathbb{N}$ with $N_\epsilon>N_\Lambda$, such that
\begin{center}
$\sum_{i=N_\epsilon+1}^\infty \Vert \exp(f^{(e_i)})\Vert_\infty<\cfrac{\epsilon}{M}$.
\end{center}
Now let 
\begin{center}
$X_c=\cup_{1\leq i\leq N_\epsilon} \pi([e_i])$. 
\end{center}
The set is compact because the space $X_v$ is compact for any $v\in V$. Now applying Definition \ref{def3} (a2), for any $n>N_\epsilon$, we have
\begin{center}
$
\begin{array}{ll}
& m_{F_n}\big(\cup_{N_\epsilon+1\leq i\leq n} \pi([e_i])\big)= \int_X \exp(f^{(e_i)}-P(F_n))dm_{F_n}\\
\leq & \sum_{N_\epsilon+1\leq i\leq n}\exp(-P(F_n))\exp(f^{(e_i)})\leq   M\sum_{N_\epsilon+1\leq i\leq n}\Vert \exp(f^{(e_i)})\Vert_\infty\\
< & M\cdot\cfrac{\epsilon}{M}=\epsilon.
\end{array}
$
\end{center}
This forces 
\begin{center}
$m_{F_n}(X_c)=m_{F_n}\Big(X\setminus \big(\cup_{N_\epsilon+1\leq i\leq n} \pi([e_i])\big)\Big)>1-\epsilon$
\end{center}
for any $n>N_\epsilon$. As $m_{F_n}(X_c)=1$ for any $N_\Lambda\leq n\leq N_\epsilon$, the tightness of the sequence follows.
\end{proof}

\begin{lemma}\label{lem14}
If a subsequence $\{m_{F_{n_j}}\}_{j\in\mathbb{N}}$ of the sequence of conformal measures $\{m_{F_n}\}_{n\geq N_\Lambda}$ converges weakly to some measure $\nu\in\mathcal{M}(X)$, then it is necessary that
\begin{center}
$\nu=m_F$.
\end{center}
\end{lemma}
\begin{proof}
We work on the pull-back measures 
\begin{center}
$\{\mu_{F_n}=m_{F_n}\circ\pi\}_{n\geq N_\Lambda}, \mu_{F}=m_{F}\circ\pi$ and $\mu=\nu\circ\pi$
\end{center}
respectively of 
\begin{center}
$\{m_{F_n}\}_{n\geq N_\Lambda}, m_{F}$ and $\nu$ 
\end{center}
on the symbolic spaces.  Note that according to \cite[Theorem 3.2.3]{MU3}, $\mu_{F_{n_j}}$ satisfies
\begin{equation}\label{eq43}
\mathcal{L}_{f_{n_j}}^*(\mu_{F_{n_j}})=e^{P(F_{n_j})}\mu_{F_{n_j}}
\end{equation}
for any $j\in\mathbb{N}$. Now we take the weak limit respectively of the two sequences of measures at both sides of (\ref{eq43}) by letting  $j\rightarrow\infty$. For the sequence at the left side, since $\{\mu_{F_{n_j}}\}_{j\in\mathbb{N}}$ converges weakly to $\mu$, according to Lemma \ref{lem11}, we have
\begin{equation}\label{eq44}
\mathcal{L}_{f_{n_j}}^*(\mu_{F_{n_j}})\stackrel{w}{\rightarrow}\mathcal{L}_f^*(\mu)
\end{equation}
as $j\rightarrow\infty$. For the sequence at the right side, since  $\lim_{j\rightarrow\infty}P(F_{n_j})=P(F)$, obviously we have
\begin{equation}\label{eq45}
e^{P(F_{n_j})}\mu_{F_{n_j}}\stackrel{w}{\rightarrow}e^{P(F)}\mu
\end{equation}
as $j\rightarrow\infty$. Now (\ref{eq44}) and (\ref{eq45}) together give that
\begin{center}
$\mathcal{L}_f^*(\mu)=e^{P(F)}\mu$.
\end{center}
Since the eigenmeasure is unique, it turns out that $\mu=\mu_{F}$, which implies
\begin{center}
$\nu=\mu\circ\pi^{-1}=\mu_{F}\circ\pi^{-1}=m_{F}$.
\end{center}

\end{proof}

Now combining  Lemma \ref{lem13}, \ref{lem14} and Corollary \ref{cor3}, we have:
\begin{theorem}\label{thm11}
For an infinite conformal-like and finitely primitive GDMS 
\begin{center}
$S=\big\{V, E, A, t, i, \{X_v\}_{v\in V}, \{s_e\}_{e\in E}, \gamma\big\}$ 
\end{center}
and a summable H\"older family of functions $F$,  the sequence of conformal measures $\{m_{F_n}\}_{n\geq N_\Lambda}$ converges weakly to the conformal measure $m_F$ as $n\rightarrow\infty$.
\end{theorem}

The above result can be strengthened to the following form.

\begin{theorem}\label{thm12}
For an infinite conformal-like and finitely primitive GDMS
\begin{center}
$S=\big\{V, E, A, t, i, \{X_v\}_{v\in V}, \{s_e\}_{e\in E}, \gamma\big\}$ 
\end{center}
and a summable H\"older family of functions $F$,  the sequence of conformal measures $\{m_{F_n}\}_{n\geq N_\Lambda}$ converges setwisely to the conformal measure $m_F$ as $n\rightarrow\infty$.
\end{theorem}
\begin{proof}
Equipped with all the proved results, the proof is similar to Proof of Theorem \ref{lem2}. We first show that 
\begin{equation}\label{eq46}
\lim_{n\rightarrow\infty} m_{F_n}(A)=m_F(A)
\end{equation}
for  $A$ being a cylinder set. We do  this by induction on the level of the cylinder sets. First we establish the initiative step. In virtue of Theorem \ref{thm11} and \cite[Theorem 2.1.5]{MU3}, for any $e\in E$, we have
\begin{equation}
\begin{array}{ll}
\lim_{n\rightarrow\infty} m_{F_n}(\pi([e])) & =\lim_{n\rightarrow\infty} \int_{X_{t(e)}} \exp(f^{(e)}-P(F_n))dm_{F_n}\\
& =\lim_{n\rightarrow\infty} \int_{X_{t(e)}} \exp(f^{(e)}-P(F))dm_{F_n}\\
 & =\lim_{n\rightarrow\infty} \int_{X_{t(e)}} \exp(f^{(e)}-P(F))dm_F\\
 & =m_F(\pi([e])).
\end{array}
\end{equation}

Now suppose (\ref{eq46}) holds on any level $k-1$ cylinder set for $k\geq 2$, we will show that it still holds on any level $k$ cylinder set. For any $k$-word $\omega=\omega_1,\omega_2,\cdots,\omega_k\in E^{k,A}$,  note that
\begin{center}
$m_{F_n}(\pi([\omega]))=\int_{\pi([\omega_2\omega_3\cdots \omega_k])}\exp(f^{(\omega_1)}-P(F_n))dm_n $
\end{center}
and
\begin{center}
$m_F(\pi([\omega]))=\int_{\pi([\omega_2\omega_3\cdots \omega_k])}\exp(f^{(\omega_1)}-P(F))dm_F $.
\end{center}
Again in virtue of Theorem \ref{thm11} and \cite[Theorem 2.1.5]{MU3}, we have
\begin{equation}\label{eq47}
\begin{array}{ll}
& \lim_{n\rightarrow\infty} \int_X \exp(f^{(\omega_1)}-P(F_n))dm_n \\
=& \lim_{n\rightarrow\infty} \int_X \exp(f^{(\omega_1)}-P(F))dm_{F_n}\\
=&\int_X \exp(f^{(\omega_1)}-P(F))dm_F.
\end{array}
\end{equation}

According to the inductive assumption on $k-1$ level cylinder sets, we have
\begin{equation}\label{eq48}
\begin{array}{ll}
\lim_{n\rightarrow\infty}m_{F_n}(\pi([\omega_2\omega_3\cdots \omega_k])) & =\lim_{n\rightarrow\infty} \int_X 1_{\pi([\omega_2\omega_3\cdots \omega_k])}dm_{F_n} \\
 & = m_F(\pi([\omega_2\omega_3\cdots\omega_k])) \\
 & = \int_X 1_{\pi([\omega_2\omega_3\cdots\omega_k])}dm_F.
\end{array}
\end{equation}
Combining (\ref{eq47}) and (\ref{eq48}) together, we have
\begin{center}
$
\begin{array}{ll}
& \lim_{n\rightarrow\infty}m_{F_n}(\pi([\omega_1\omega_2\cdots \omega_k]))\\
=& \lim_{n\rightarrow\infty} \int_X \exp(f^{(\omega_1)}-P(F_n))\cdot 1_{\pi([\omega_2\omega_3\cdots \omega_k])} dm_{F_n} \\
= & \lim_{n\rightarrow\infty} \int_X \exp(f^{(\omega_1)}-P(F))\cdot 1_{\pi([\omega_2\omega_3\cdots \omega_k])}dm_{F_n} \\
= & \int_X \exp(f^{(\omega_1)}-P(F))\cdot 1_{\pi([\omega_2\omega_3\cdots \omega_k])}dm_F\\
= & m_F(\pi([\omega_1\omega_2\cdots \omega_k])),
\end{array}
$
\end{center}
which completes the inductive step. 

Finally, we can first extend (\ref{eq46}) from cylinder sets to closed measurable ones, and then to all measurable  ones via regularity of the measures, or we can directly reach the conclusion in virtue of \cite[Theorem 2.3]{FKZ}. 

\end{proof}

Note that we have in fact proved the following result simultaneously.

\begin{corollary}
For an infinite conformal-like and finitely primitive GDMS and a summable H\"older family of functions $F$,  the sequence of conformal measures $\{\mu_{F_n}\}_{n\geq N_\Lambda}$ converges setwisely to the conformal measure $\mu_F$ as $n\rightarrow\infty$ on the symbolic space.
\end{corollary}

By a similar argument as to prove Theorem \ref{thm10}, we can show that, under the context of Theorem \ref{thm12}, we have the following result.
\begin{proposition}
Any of the four sequences of measures 
\begin{center}
$\{m_{F_n}\}_{n=N_\Lambda}^\infty$, $\{m_{F_n}^*\}_{n=N_\Lambda}^\infty$, $\{\mu_F\}_{n=N_\Lambda}^\infty$ and $\{\mu_{F_n}^*\}_{n=N_\Lambda}^\infty$
\end{center}
does not converge under the TV topology on $\mathcal{M}(J)$, because
\begin{center}
$\Vert m_{F_{n_1}}-m_{F_{n_2}}\Vert_{TV}=1$
\end{center}
and 
\begin{center}
$\Vert \mu_{F_{n_1}}-\mu_{F_{n_2}}\Vert_{TV}=1$
\end{center}
for any two integers $N_\Lambda\leq n_1<n_2\in\mathbb{N}$.
\end{proposition}

For a measure $\nu\in\mathcal{M}(E^{\infty,A})$, still let $h_\nu(\sigma)$ and $\lambda_\nu(\sigma)$ be its entropy and average Lyapunov exponent.  Theorem \ref{thm12} induces the following result on dimensions of these measures, regardless of whether $h_{\mu_F}(\sigma)=\infty$ or $<\infty$.
\begin{theorem}\label{thm13}
For an infinite conformal-like and finitely primitive GDMS
\begin{center}
$S=\big\{V, E, A, t, i, \{X_v\}_{v\in V}, \{s_e\}_{e\in E}, \gamma\big\}$ 
\end{center}
and a summable H\"older family of functions $F$,  we have
\begin{center}
$dim_H^0 m_F=dim_H^1 m_F=dim_H^0 m_F^*=dim_H^1 m_F^*=\lim_{n\rightarrow\infty}\cfrac{h_{\mu_{F_n}^*(\sigma)}}{\lambda_{\mu_{F_n}^*(\sigma)}}$.
\end{center}
\end{theorem}
\begin{proof}
Considering all the established results Corollary \ref{cor2}, \cite[Lemma 3.1(3)]{HS} and Theorem \ref{thm12}, the proof of the above equalities except the last one is of the same as Proof of Theorem \ref{thm3}. The last equality follows from the above results and \cite[Corollary 4.4.5]{MU3}, since the corresponding ergodic measure $\mu^*_{F_n}$ originated from each truncated finite sub-GDMS $S_n$ has finite entropy for any $n\geq N_\Lambda$.
\end{proof}
Compare the result with \cite[Corollary 4.4.5]{MU3}.

Now we apply the above results for infinite conformal-like and finitely primitive GDMS with summable H\"older family of functions to some kind of CGDMS-the regular finitely primitive CGDMS (there is no need to emphasize conformal-like as any CGDMS is conformal-like).  For any $t\geq 0$, consider the family of functions
\begin{center}
$F_t=\{\log |s_e'|^t\}_{e\in E}$.
\end{center}
It is a H\"older  family of order $-\log\gamma$.   Let $P(t)=P(F_t)$. We collect some results from \cite[Chapter 4]{MU3} in the following theorem.
\begin{theorem}[Mauldin-Urba\'nski]
For a finitely primitive CGDMS, a $t$-conformal measure $\mathbf{m}$ exists if and only if $P(t)=0$. In case of existence of such a measure, we have
\begin{center}
$\mathbf{m}=m_{F_t}$
\end{center}
and
\begin{center}
$t=HD(J)=h$.
\end{center}
\end{theorem}

We call a CGDMS \emph{regular} if $P(t)=0$ admits a solution ($P(t)$ admits at most one zero), or else we call it \emph{irregular}. From now on we always assume our infinite CGDMS is finitely primitive and regular. In virtue of Lemma \ref{lem12}, there exists a unique $h_n$-conformal measure  $\mathbf{m}_n$ for the finite sub-CGDMS 
\begin{center}
$S_n=\big\{V_n, E_n, A|_{E_n\times E_n}, t, i, \{X_v\}_{v\in V_n}, \{s_e\}_{e\in E_n}, \gamma\big\}$
\end{center}
for any $n\geq N_\Lambda$. Similarly, let $\mathbf{m}^*, \{\mathbf{m}_n^*\}_{n\geq N_\Lambda}$ be the equivalent ergodic measures respectively to the conformal ones $\mathbf{m}, \{\mathbf{m}_n\}_{n\geq N_\Lambda}$, and 
\begin{center}
$\bm \mu, \{\bm \mu_n\}_{n\geq N_\Lambda}, \bm \mu^*, \{\bm \mu_n^*\}_{n\geq N_\Lambda}$ 
\end{center}
be the  pull-back measures of the corresponding ones under $\pi$ on the symbolic spaces.

Now apply our Theorem \ref{thm12} to the infinite finitely primitive regular CGDMS with the summable H\"older  family $F_h$, we get the following result.
\begin{corollary}
For an infinite finitely primitive regular CGDMS, we have 
\begin{center}
$\mathbf{m}_n\stackrel{s}{\rightarrow}\mathbf{m}$
\end{center}
as $n\rightarrow\infty$.
\end{corollary}

Applying our Theorem \ref{thm13} to the finitely primitive regular CGDMS with the summable H\"older  family $F_h$, we get Theorem \ref{cor4} instantly, regardless of whether $h_{\mu}(\sigma)=\infty$ or $<\infty$.

\begin{rem}\label{rem2}
According to \cite[Remark 3.2]{CLU}, if 
\begin{center}
$\#V=1$ 
\end{center}
and 
\begin{center}
$A_{ee'}=1$ if and only if $t(e)=i(e')$ for any two edges $e,e'\in E$, 
\end{center}
the CGDMS degenerates into a CIFS, which is always finitely primitive. So the last equality of (\ref{eq42}) follows from the last equality of Theorem \ref{cor4}. 
\end{rem}

As a final result of this section, similar to the case of CIFS, considering the correlation dimension and the modified correlation dimension of the above measures, we have the following result. 

\begin{theorem}\label{thm19}
For any finitely primitive regular CGDMS, we have
\begin{center}
$dim_C \mathbf{m}=dim_C \mathbf{m}^*=dim_{MC} \mathbf{m}=dim_{MC} \mathbf{m}^*=HD(J)=h$.
\end{center}

\end{theorem} 
\begin{proof}
Equipped with \cite[Theorem 4.2.11]{MU3}, this is of similar process as the Proof of Theorem \ref{thm14}.
\end{proof}

\section{More applications and some notes on the dimension of a measure and its logarithmic density}\label{sec5}

The applications of our techniques are still not limited to the above mentioned IFS.  Besides the applications to the CIFS and CGDMS, our  continuity results in fact can be applied to more kinds of infinite IFS in which the shift map (or the induced map on the ambient space $X$) has finite or infinite entropy with respect to corresponding measures in ones' concerns. For example, the infinite \emph{parabolic} IFS (with overlaps) in \cite{SSU1, SSU2}, the infinite \emph{random conformal graph directed Markov systems} (RCGDMS) in \cite{RU1}, the infinite \emph{weakly geometrically stable} IFS in \cite{Kae}, the infinite \emph{weakly conformal} IFS (with overlaps) in \cite{FH}, $\cdots$. Results on dimensions of measures originated from these systems are in vision.

We choose to apply our continuity results to dimensions of measures on spaces with IFS structure on it, however, these results  also shed some lights on dimensions of measures on spaces with some general dynamical structures. For example, as a tool to tackle the obstacle of exploding entropy of some measure with respect to some partitions on the $X$ space, we have the following result, as a supplement of \cite[Main Theorem]{You}.
\begin{corollary}\label{cor8}
Let $X$ be a (non-compact) two Riemannian manifold, $T: X\rightarrow X$ is a $C^2$ endomorphism. For an ergodic measure $\nu\in\mathcal{M}(X)$ with respect to $T$, if one can find a sequence of ergodic measures $\{\nu_n\in\mathcal{M}(X)\}_{n\in\mathbb{N}}$  with respect to a corresponding sequence of $C^2$ endomorphisms $\{T_n: X\rightarrow X\}_{n\in\mathbb{N}}$  such that
\begin{center}
$\nu_n\stackrel{s}{\rightarrow}\nu$ 
\end{center}
as $n\rightarrow\infty$ and $0\leq h_{\nu_n}(T_n), \lambda_1(T_n), \lambda_2(T_n)<\infty$ 
for any $n\in\mathbb{N}$, then
\begin{center}
$dim_H^0 \nu=dim_H^1 \nu=\lim_{n\rightarrow\infty} dim_H^0 \nu_n=\lim_{n\rightarrow\infty}  dim_H^1 \nu_n=\lim_{n\rightarrow\infty} h_{\nu_n}(T_n)\Big(\cfrac{1}{\lambda_1(T_n)}-\cfrac{1}{\lambda_2(T_n)}\Big)$,
\end{center}
in which $\lambda_1(T_n)\geq \lambda_2(T_n)$ represent respectively the two Lyapunov exponent of the map $T_n$ and $T_n^{-1}$ (they are constants for \emph{a.e.} $x\in X$ due to ergodicity of $T_n$).
\end{corollary} 
\begin{proof}
This follows from \cite[Main Theorem]{You}, Young's Lemma and our two semi-continuity results Theorem \ref{thm1} and \ref{thm4}.
\end{proof}

We focus on the measure-dimension mappings $dim_H^0, dim_H^1, dim_C, dim_{MC}$ in this work, but it would be interesting to ask the continuity question on other measure-dimension mappings, under suitable topology on $\mathcal{M}(X)$ (may not be restricted to the three kinds of topology we rely on). For example, the (\emph{upper}- and \emph{lower}-) $L^q$ \emph{dimension} (see for example \cite{Ren}), the (\emph{upper}- and \emph{lower}-) \emph{entropy dimension} (see for example \cite[(1.2)]{PerS} or \cite[2.9]{BHR}) and  the \emph{lower} and \emph{upper} \emph{box-counting} or \emph{packing} measure-dimension mappings defined as following:
  
\begin{center}
$dim_B^0 \nu=\inf\{BD(A): \nu(A)>0, A\in\mathscr{A}\}$,
\end{center}
\begin{center}
$dim_P^0 \nu=\inf\{PD(A): \nu(A)>0, A\in\mathscr{A}\}$,
\end{center}
and
\begin{center}
$dim_B^1 \nu=\inf\{BD(A): \nu(A)=1, A\in\mathscr{A}\}$,
\end{center}
\begin{center}
$dim_P^1 \nu=\inf\{PD(A): \nu(A)=1, A\in\mathscr{A}\}$,
\end{center}
in which $BD(A)$ and $PD(A)$  represent respectively the box-counting and packing dimension of a set $A$.

For an arbitrary finite measure $\nu$ on $X$ without dynamical structure, the lower and upper logarithmic density of the measure $\nu$ give upper and lower bounds on $dim_H^1 \nu$ and $dim^0_H \nu$. Now let 
\begin{center}
$\underline{\gamma}=\sup \big\{a\geq 0: \nu\{x\in X:\underline{d}(\nu,x)>a\}>0\big\}$
\end{center}
and 
\begin{center}
$\overline{\gamma}=\inf \big\{a\geq 0: \nu\{x\in X:\overline{d}(\nu,x)<a\}>0\big\}$.
\end{center}
According to \cite[Proposition 2.1]{You}, it is easy to see that $dim_H^1 \nu$ and $dim^0_H \nu$ satisfies 
\begin{equation}\label{eq15}
\underline{\gamma}\leq dim^0_H \nu\leq dim_H^1 \nu\leq \overline{\gamma}.
\end{equation}

In many cases, we could expect that $dim_H^0 \nu$ or $dim^1_H \nu$ can take the value $\overline{\gamma}$ or $\underline{\gamma}$ . For example, if we have 
\begin{center}
$\nu\{x\in X: d(\nu,x)=\underline{\gamma}\}>0$,
\end{center}
or there exist two decreasing sequences $\{\underline{\gamma}_{n,1}\}_{n=1}^\infty$ and $\{\underline{\gamma}_{n,2}\}_{n=1}^\infty$ with 
\begin{center}
$\underline{\gamma}_{n,1}<\underline{\gamma}_{n,2}$ for any $n\in\mathbb{N}$ and $\underline{\gamma}_{n,1}, \underline{\gamma}_{n,2}\rightarrow \underline{\gamma}$ as $n\rightarrow\infty$,
\end{center}
 such that
\begin{center}
$\nu\{x\in X: \underline{\gamma}_{n,1}<\underline{d}(\nu,x)\leq\overline{d}(\nu,x)<\underline{\gamma}_{n,2}\}>0$,
\end{center}
then  $dim^0_H \nu=\underline{\gamma}$. Examples in which the inequality $dim^0_H \nu<\overline{\gamma}$ or $dim_H^1 \nu>\underline{\gamma}$ in (\ref{eq15}) are easy to get. 
 
\begin{exm}\label{exm2}
Let $\nu$ be a finite measure on $[0,1]$ with $d(\nu,x)=c$ a.e. on $[0,1]$, let $\nu'$ be a finite measure on $[1,2]$ with $d(\nu',x)=c'>c$ a.e. on $[1,2]$. Consider the measure $\nu''$ on $[0,2]$ to be
\begin{center}
$\nu''(A)=\nu(A\cap[0,1])+\nu'(A\cap[1,2])$ 
\end{center}
for any measurable set $A\subset[0,2]$.
\end{exm}
It is easy to see that, in Example \ref{exm2}, we have $\underline{\gamma}=c$ and $\overline{\gamma}=c'$. Moreover, we have
\begin{center}
$dim^0_H \nu=\underline{\gamma}=c<\overline{\gamma}$ 
\end{center}
and  
\begin{center}
$dim_H^1 \nu=\overline{\gamma}=c'>\underline{\gamma}$. 
\end{center}
We are not able to give an example such that the strict inequality  

\begin{equation}\label{eq16}
dim^0_H \nu>\underline{\gamma}
\end{equation}
 or
\begin{equation}\label{eq17}
dim_H^1 \nu<\overline{\gamma}
\end{equation}
holds, simultaneously or separately. Thus we would like to pose it as an open problem here.

\end{document}